\documentclass[12pt]{amsart}
\usepackage{latexsym,amssymb,amsmath,amsthm}
\usepackage{comment}
\usepackage{float}
\usepackage{graphicx}
\newtheorem{thm}{Theorem}

\newtheorem{lemma}[thm]{Lemma}

\theoremstyle{definition}

\def\question#1{\begin{center}\fbox{\parbox{10cm}{{\bf Question: }#1}}\end{center}}

\title{On quadratic orbital networks}
\author{Oliver Knill}
\date{December 1, 2013}
\address{
        Department of Mathematics \\
        Harvard University \\
        Cambridge, MA, 02138
        }

\begin{document}
\maketitle
\begin{abstract}
These are some informal remarks on quadratic orbital networks over finite fields $Z_p$. 
We discuss connectivity, Euler characteristic, number of cliques, planarity, 
diameter and inductive dimension. We prove that for $d=1$ generators, the Euler characteristic
is always nonnegative and for $d=2$ and large enough $p$ the Euler characteristic is negative. 
While for $d=1$, all networks are planar, we suspect that for $d \geq 2$ and large enough $p$, all
networks are non-planar. As a consequence on bounds for the number of complete subgraphs of a
fixed dimension, the inductive dimension of all these networks goes $1$ as $p \to \infty$. 
\end{abstract}

\section{Polynomial orbital networks}

Given a field $R=Z_p$, we study orbital graphs $G =(V,E)$ defined by polynomials $T_i$
which generate a monoid $T$ acting on $R$. We think of $(R,T)$ as a {\bf dynamical system} where
positive {\bf time} $T$ is given by the monoid of words $w = w_1 w_2 \dots w_k$ 
using the generators $w_k \in A=\{T_1, \dots ,T_d \; \}$ as alphabet and where $\{ T^wx \; | \; w \in R \; \}$
is the {\bf orbit of $x$}. The {\bf orbital network} \cite{GK1,KnillAffine,KnillMiniatures}
is the finite simple 
graph $G$ where $V=R$ is the set of vertices and where two vertices $x,y \in V$ are connected if there 
exists $T_i$ such that $T_i(x)=y$ or $T_i(y)=x$. The network generated by the system consists of the union
of all orbits. As custom in dynamics, one is interested in invariant 
components of the system and especially forward attractors $\Omega(x)$ of a point $x$
as well as the {\bf garden of eden}, the set of points which are not in the image of any $T_i$.
We are also interested in the {\bf inductive dimension} of the network. This relates to the existence and number of 
cliques, which are complete subgraphs of $G$. \\

Since most questions one can ask are already unsettled for quadratic polynomial maps, 
we restrict here to polynomial maps of the form $T_i(x) = x^2+a_i$ and call the corresponding finite
simple graphs {\bf quadratic orbital networks}. These systems have been studied since a while \cite{Robert}.
For the case $T(x)=x^2$ especially, see \cite{SomerKrizek}.
For a given prime $p$ and a fixed number of generators $d \geq 1$ we have a natural probability space 
$X_p^d$ of all ordered $d$ tuples $a_1<a_2< \dots < a_d$ each generating an orbital network. 
We can now not only study properties for individual networks but the probability that some 
event happens and especially the asymptotic properties in the limit $p \to \infty$. \\

\begin{figure}[H]
\scalebox{0.14}{\includegraphics{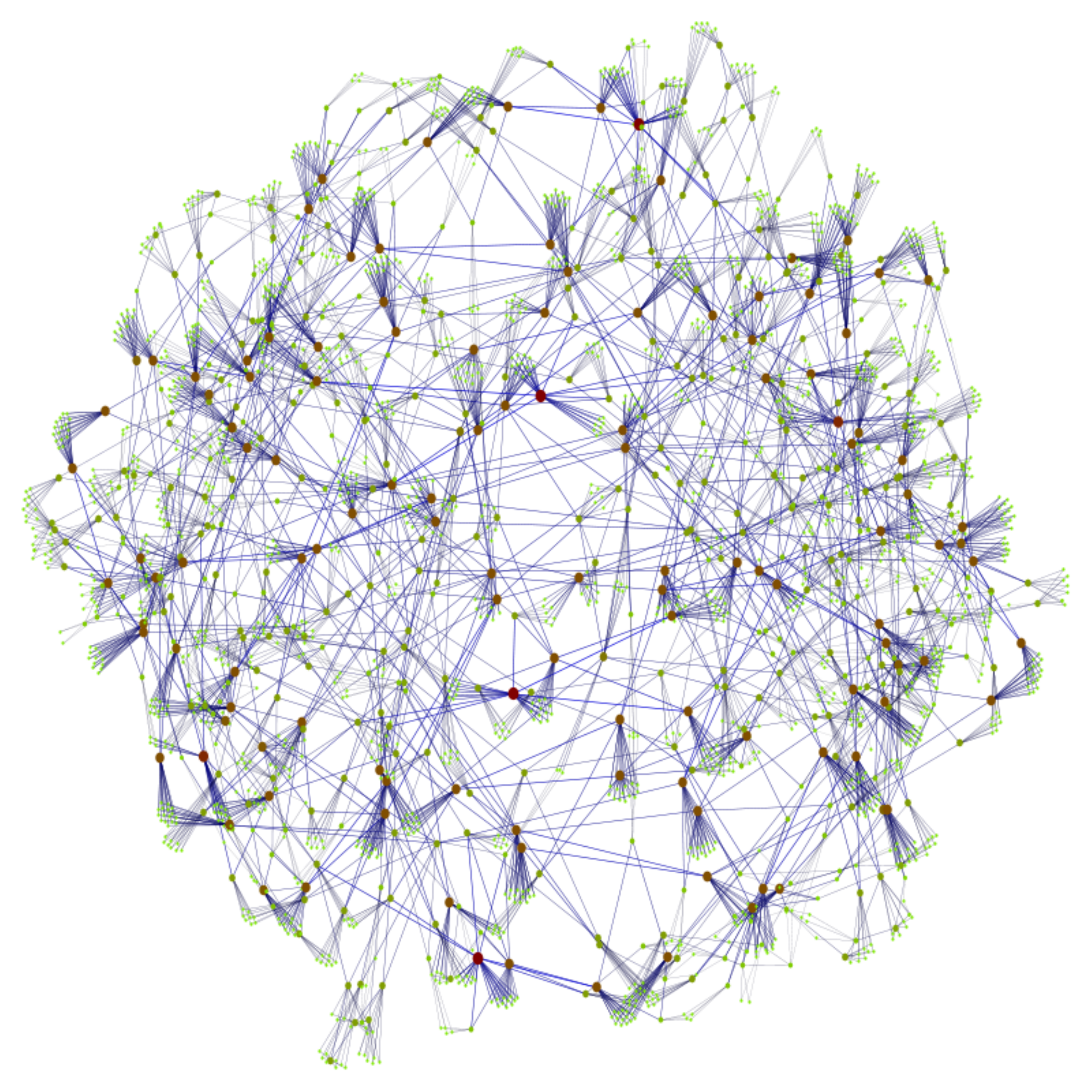}}
\scalebox{0.14}{\includegraphics{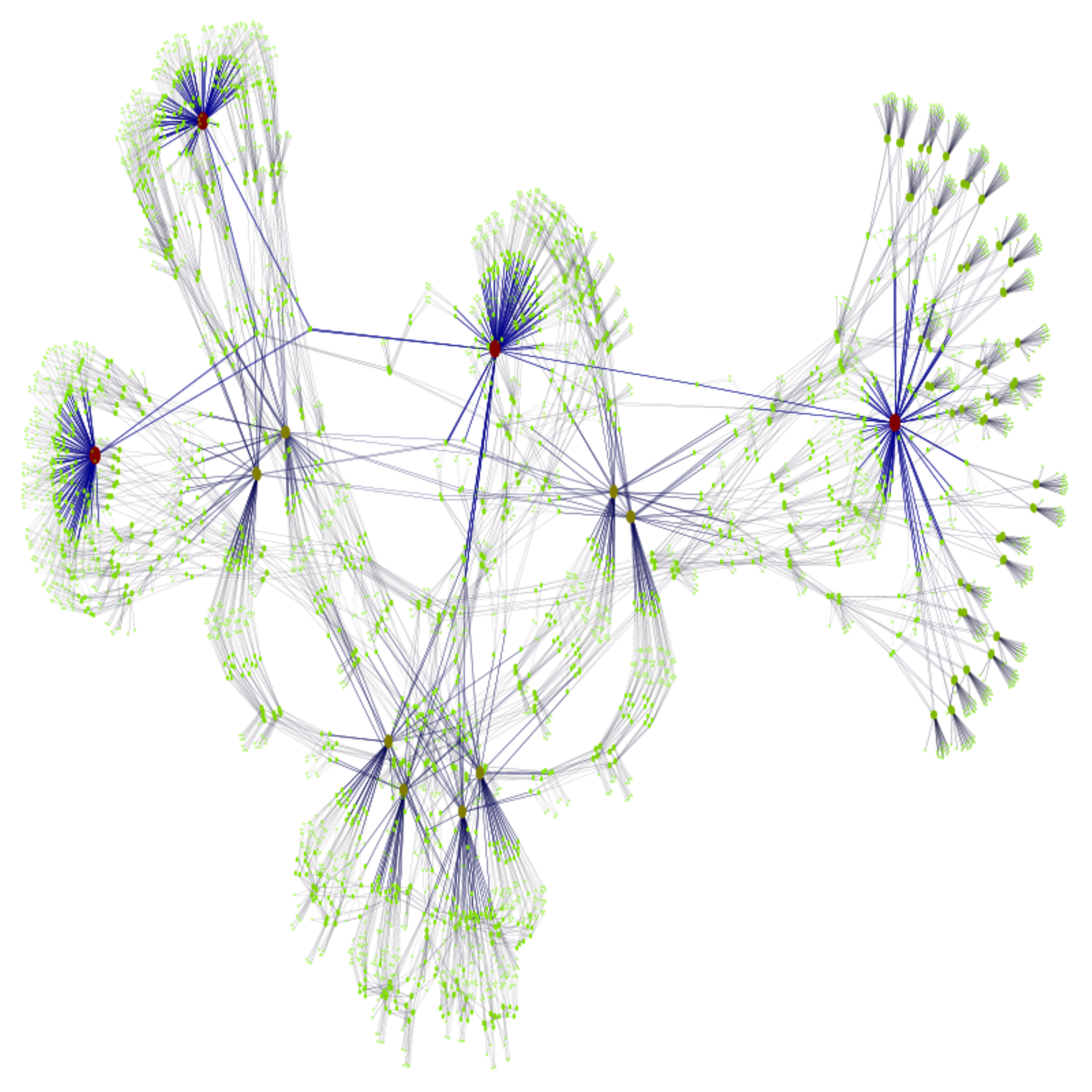}}
\scalebox{0.14}{\includegraphics{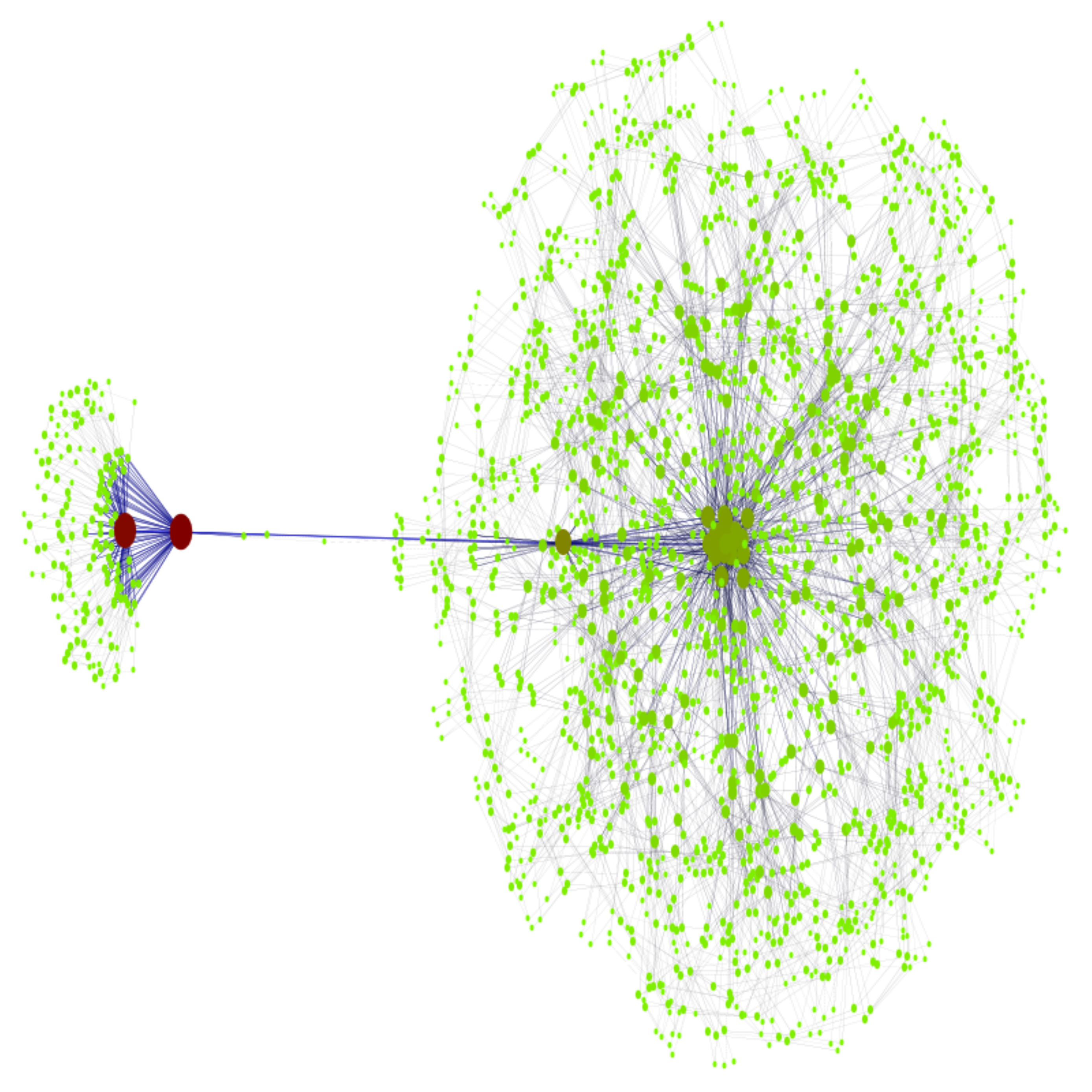}}
\caption{
The orbital network generated by $T(x)=x^2-1,S(x)=x^2+1$ on $Z_{2310}$, 
on $Z_{4096}$ and $T(x)=x^2-1$ and $S(x)=x^2$ on $Z_{2187}$. The rings
$Z_n$ are not field but rather consist of smooth $n$ which leads 
to special structures. 
} \end{figure}

Many of the questions studied here have been looked at in other contexts. 
The dynamical system given by $T(x)=x^2$ on the field $Z_p$ for example was
well known already to Gauss, Euler, Fermat and their contemporaries.
What we call the {\bf garden of eden} is in this case the set 
of {\bf quadratic non-residues}. As Gauss knew already, half of the vertices different from $0$ are there.
The dynamical system has $T(x)=x^2$ has a single component if and only if $p$ is a Fermat
prime and two components $\{0\},Z_p^*$ if and only if $2$ is a {\bf primitive root} modulo $p$. 
One believes that asymptotically $37.4...$ percent of the primes $p$ have $2$ as a primitive root
leading to connectedness on $Z_p^*$. 
This probability is related to the {\bf Artin constant} $\prod_p (1-1/(p(p-1))) = 0.3739558...$. \\

\begin{figure}
\scalebox{0.14}{\includegraphics{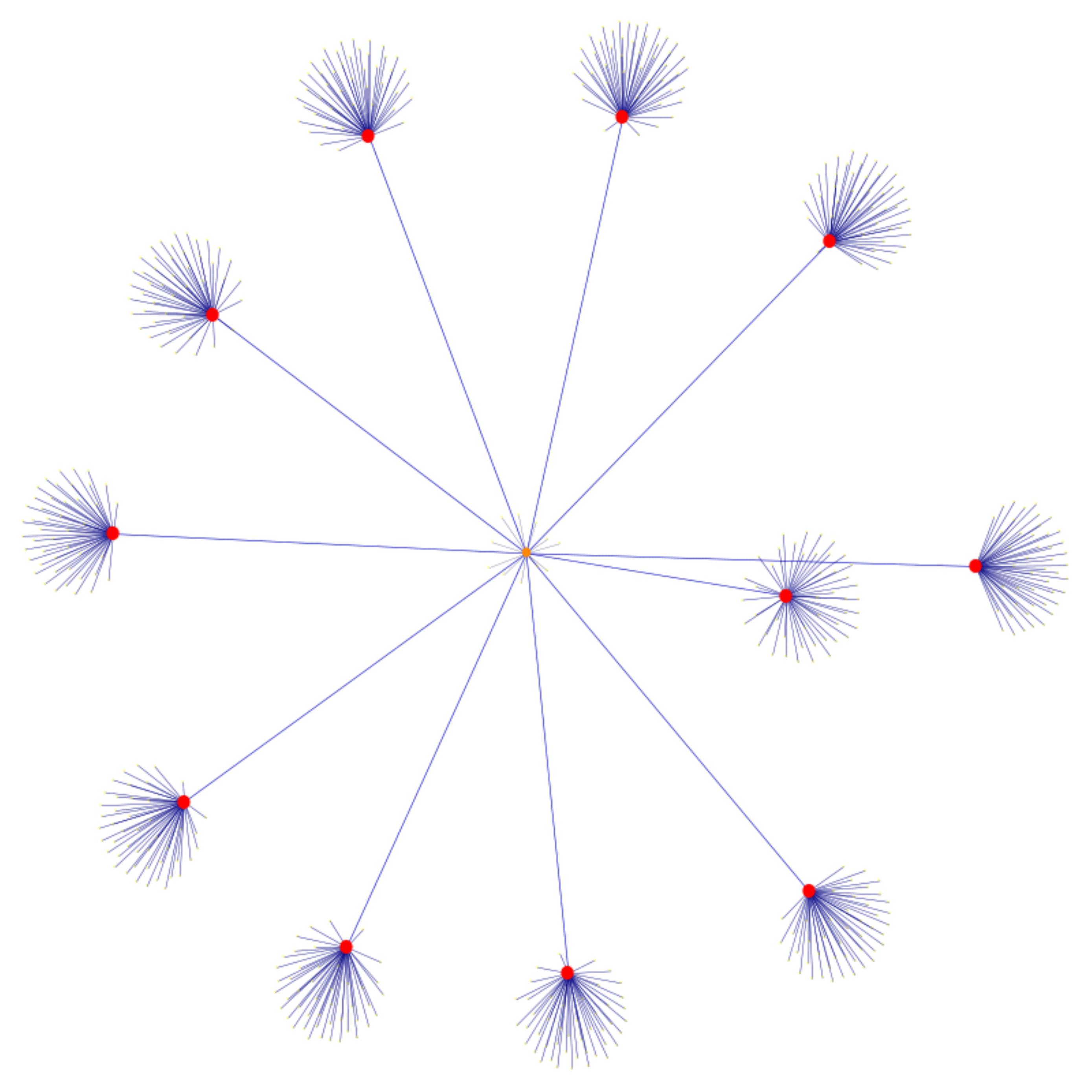}}
\scalebox{0.14}{\includegraphics{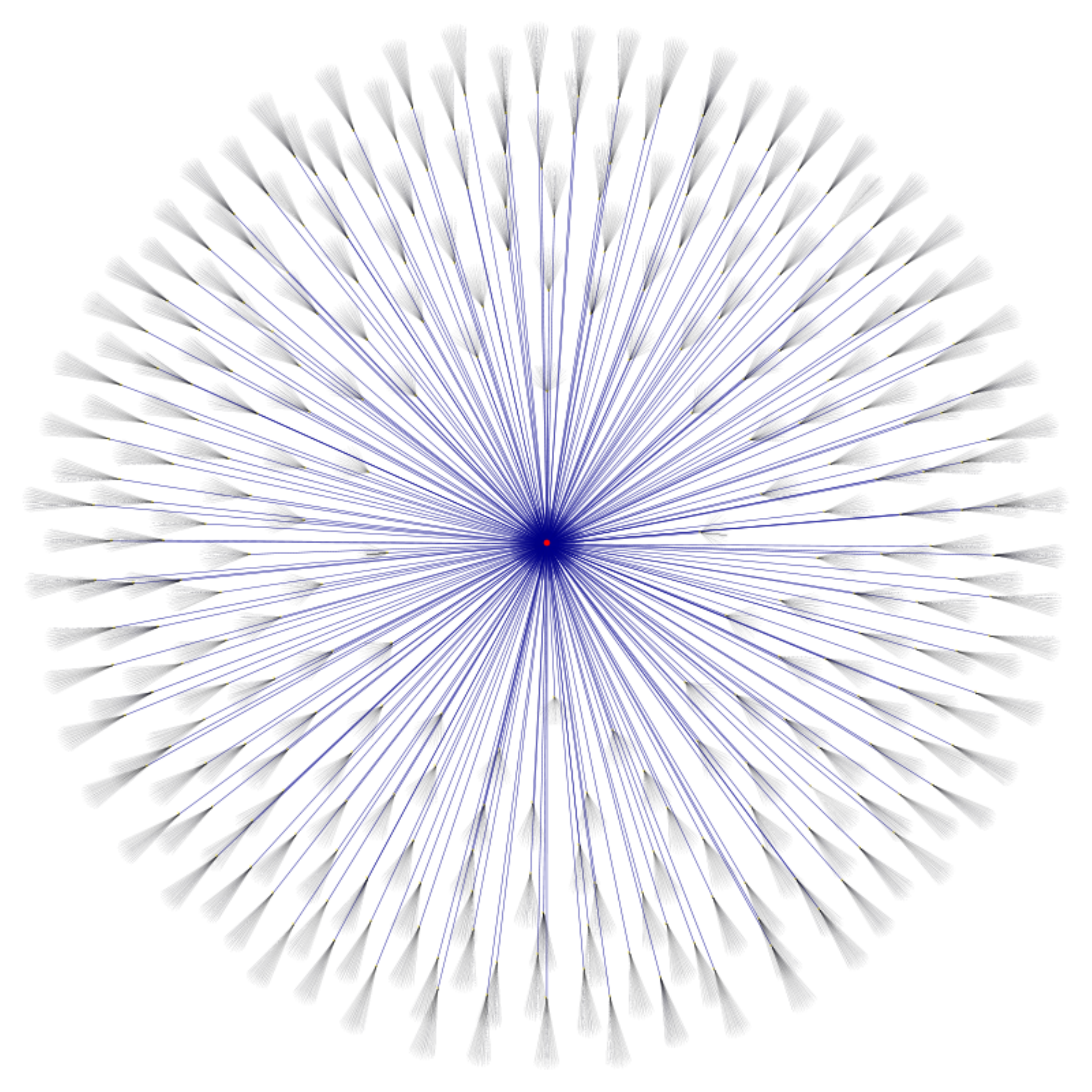}}
\scalebox{0.14}{\includegraphics{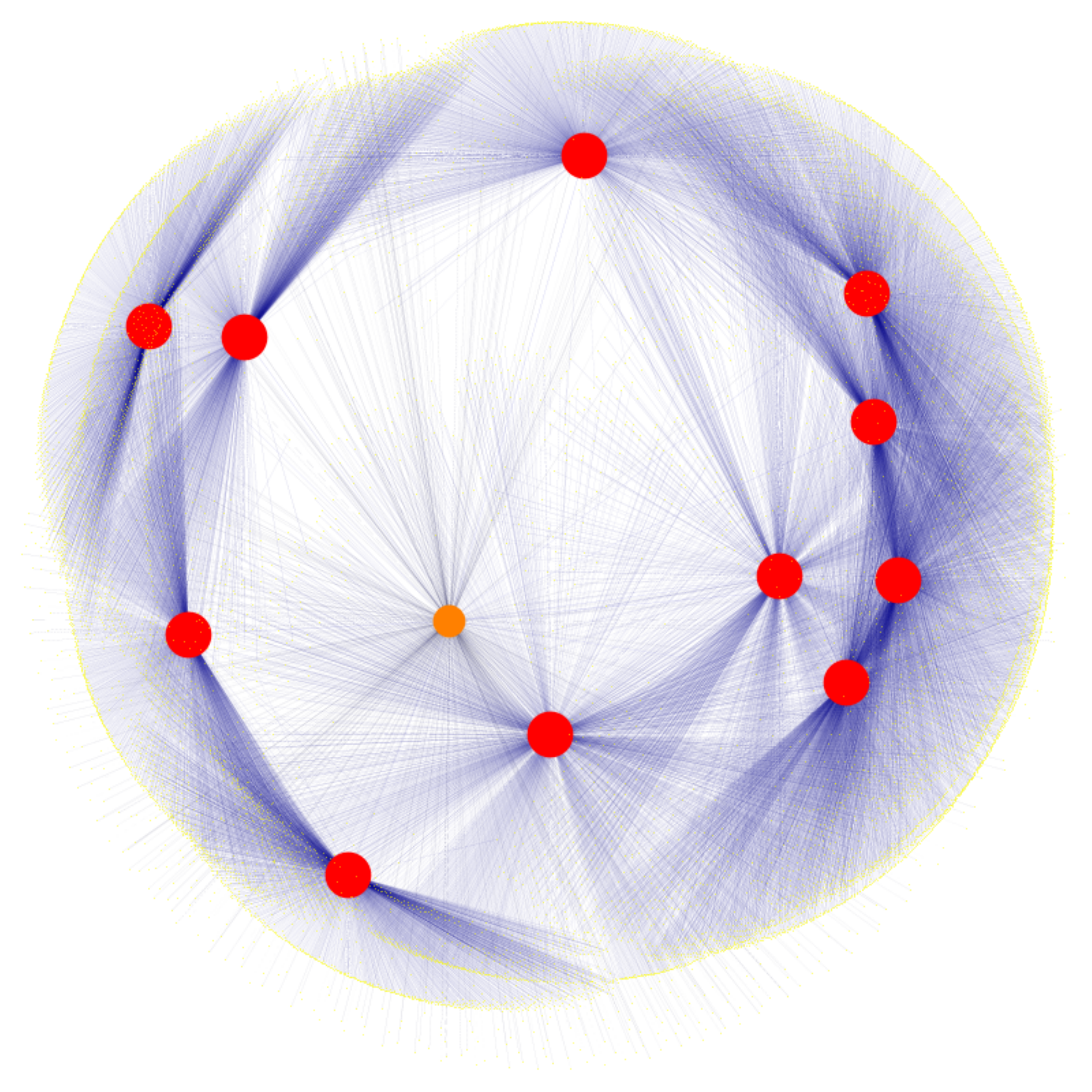}}
\caption{
The orbital network generated by $T(x)=p x^2$ on $Z_{p^2}$ with $p=23$ 
and $T(x)=p^2 x^2$ on $Z_{p^3}$. 
} \end{figure}

We have worked with affine maps in \cite{KnillAffine}. Even very simple questions are interesting. 
Working in the ring $Z_{p^2}$ for example, with linear maps $T(x)=ax$ is interesting because solutions
to $a^{p-1}=1$ modulo $p^2$ are called {\bf Fermat solutions} \cite{Hua}
There is still a lot to explore for affine maps but the complexity we see for quadratic maps is 
even larger. \\

Understanding the  system $T(x)=x^2$ on the ring $F_{pq}$ for two primes $p,q$ is the 
{\bf holy grail} of integer factorization as Fermat has known already. Finding a second root of a number
$a^2$ is equivalent to factorization as virtually all advanced factorization methods like the continued fraction, 
the Pollard rho or the quadratic sieve method make use of: if $b^2=a^2$, then ${\bf gcd}(b-a,n)$ reveals one of the factors $p$ or $q$. 
The fact that integer factorization is hard shows that understanding even the system
$T(x)=x^2$ is difficult if $n$ is composite. In fact, it is difficult even to find the second square root of $1$. If 
one could, then factorization would be easy. \\

\begin{figure}[H]
\scalebox{0.13}{\includegraphics{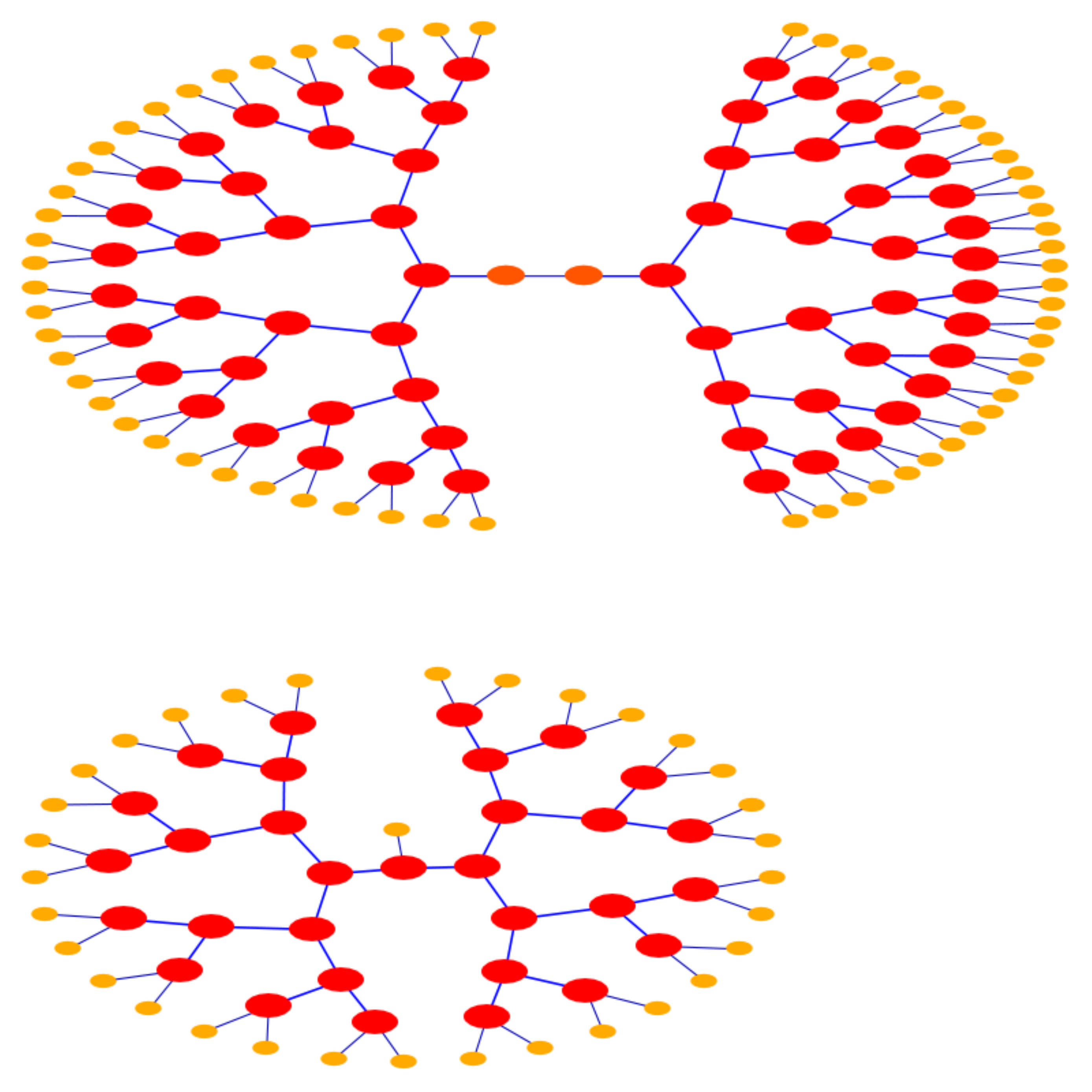}}
\scalebox{0.13}{\includegraphics{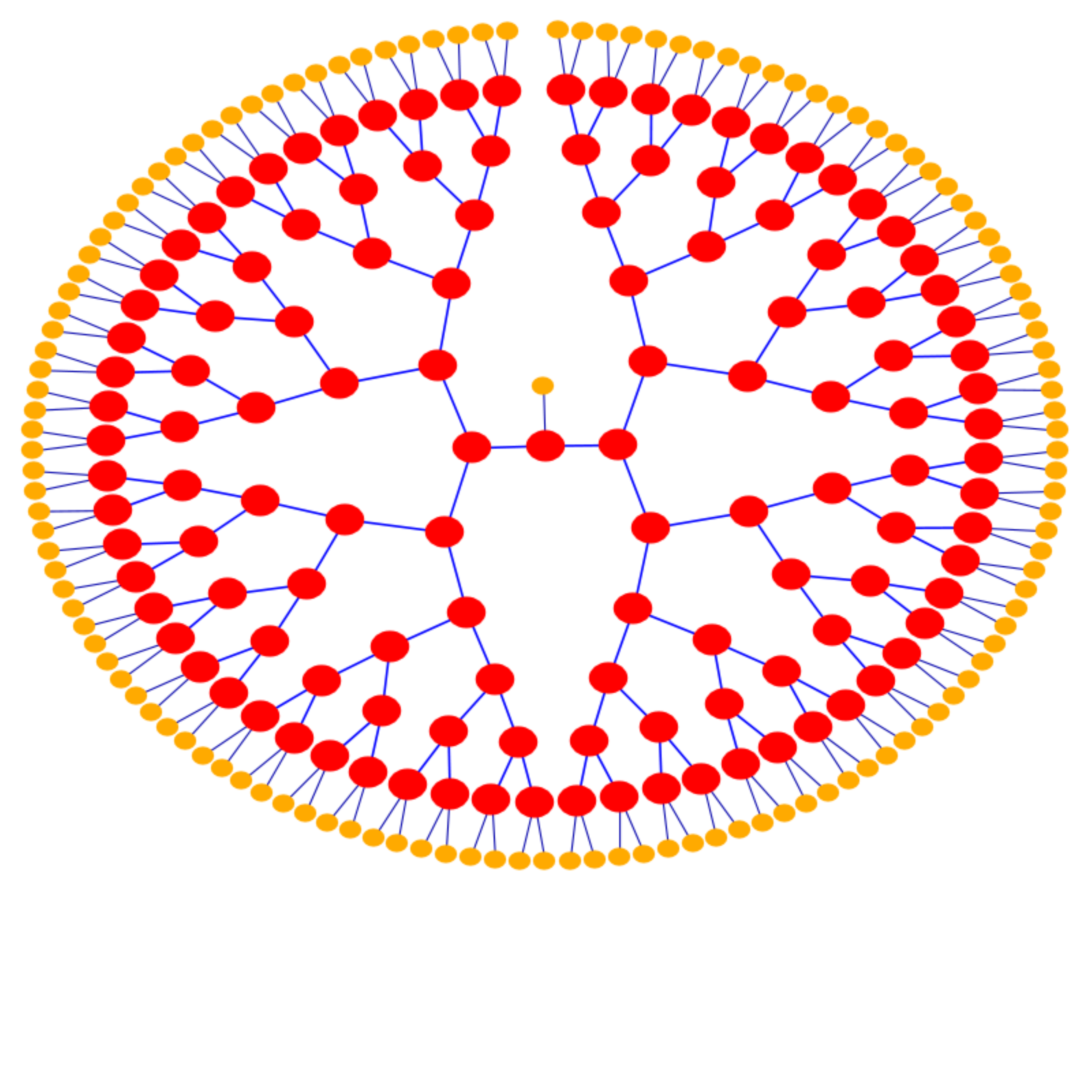}}
\scalebox{0.13}{\includegraphics{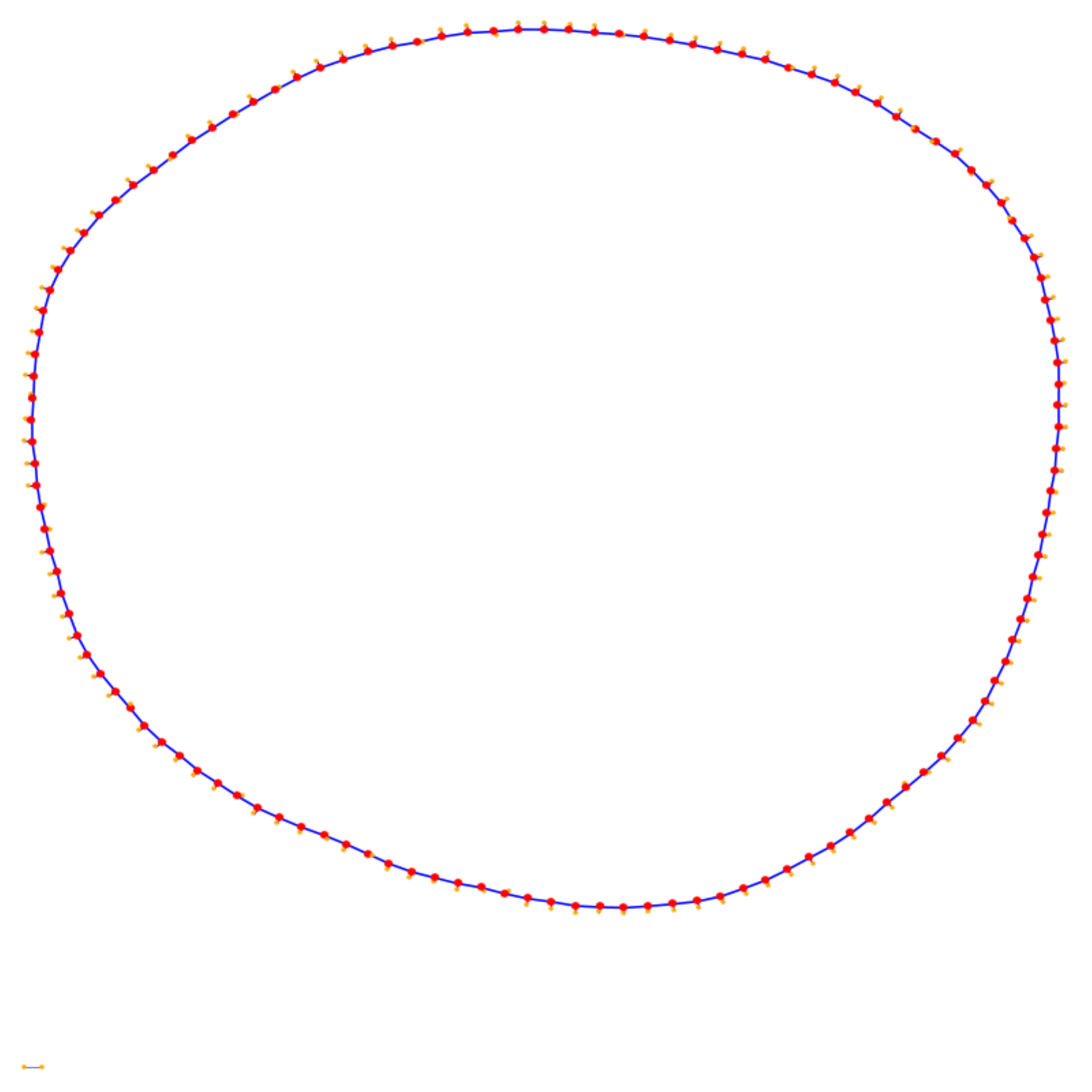}}
\caption{
The orbital network to $T(x)=x^2$ on $Z_p$. In the first
case, with $p=193$, we have two connected components
for $p=257$, a Fermat prime, we have one tree.
For $p=263$, there is one ring and one isolated vertex $0$. 
This always happens if $2$ is a primitive root modulo $p$. 
} \end{figure}

\begin{figure}[H]
\scalebox{0.3}{\includegraphics{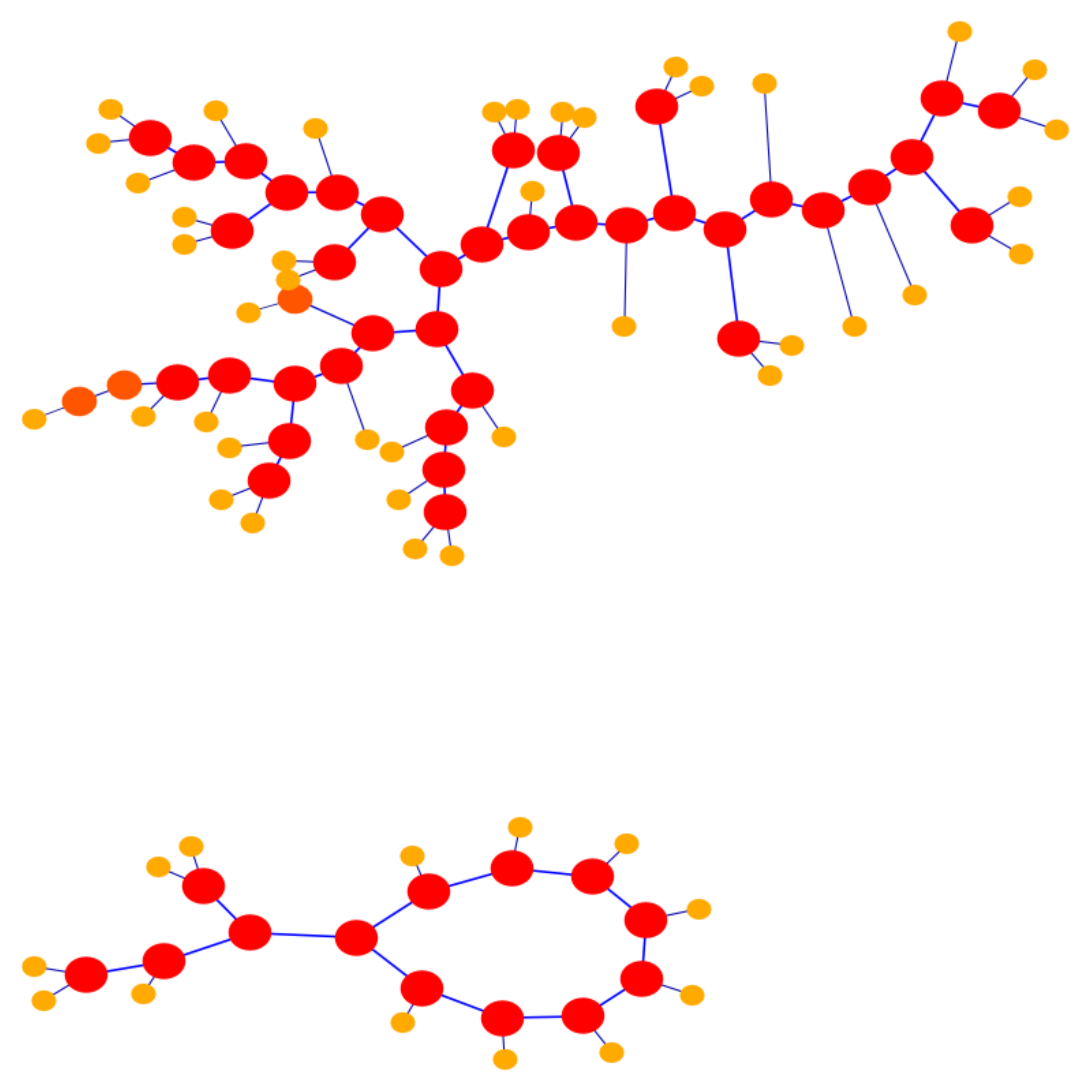}}
\caption{
The orbital network to $T(x)=x^2+1$ on $Z_{107}$ has two components. 
The smaller one 
has Euler characteristic $0$
and the larger one 
is a tree. For any network generated
by one map, every connected component
contains maximally one closed loop to which many transient 
trees lead. 
} \end{figure}

While for $d=1$, connected quadratic orbitals become rare in the limit $p \to \infty$ and 
for $d=2$, the probability goes to $1$, it is difficult to find examples of 
disconnected quadratic orbital graphs of $d=3$.
Initial experiments made us believe that there are none. In the mean time we found one. 

\begin{figure}[H]
\scalebox{0.13}{\includegraphics{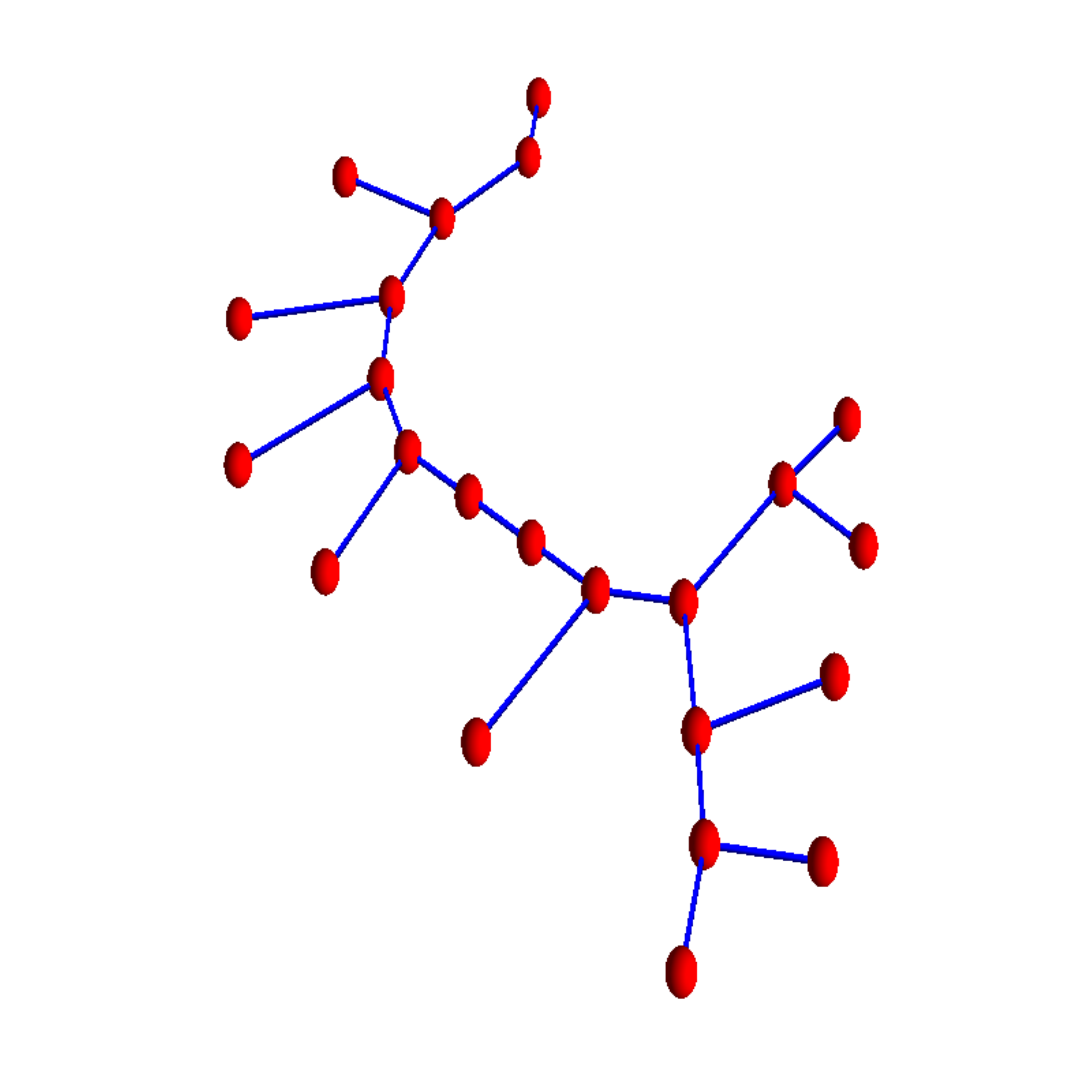}}
\scalebox{0.13}{\includegraphics{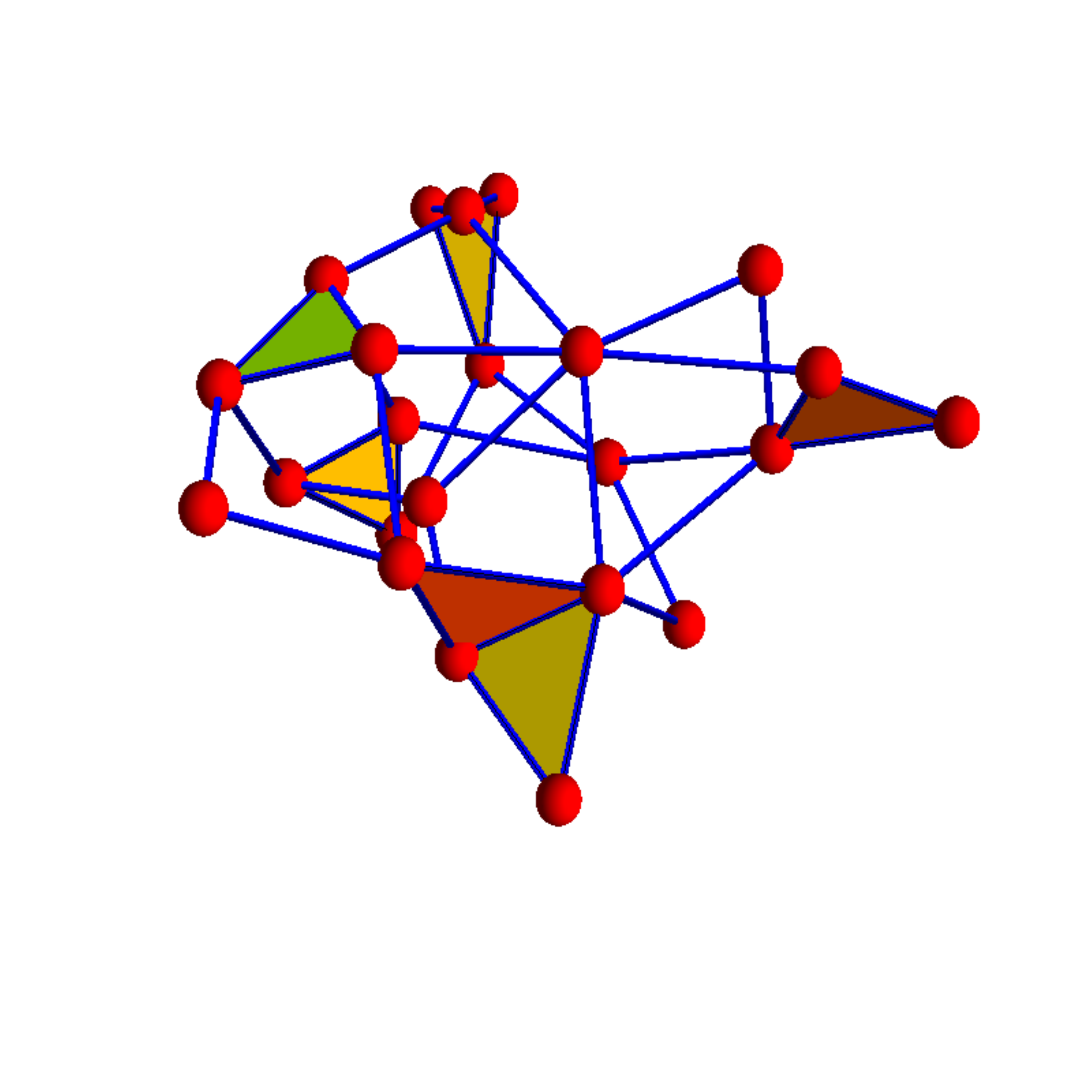}}
\scalebox{0.13}{\includegraphics{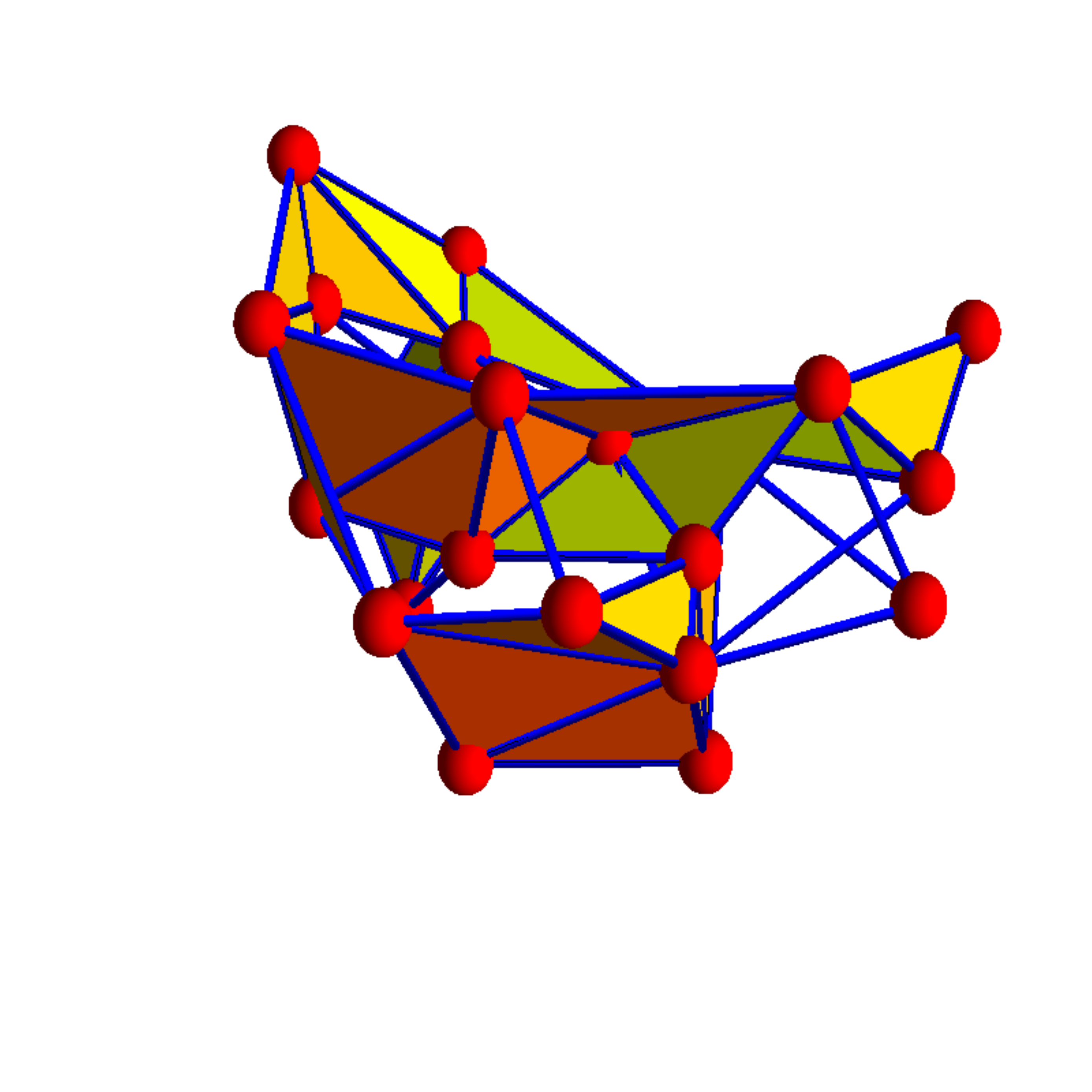}}
\caption{
Orbital networks on $Z_{23}$ generated by one, two
or three transformations $T_k(x)=x^2+k$.
} \end{figure}

It is still possible that there is a largest 
prime for which they still exist and that all quadratic graphs with $3$ different generators 
are connected if $p$ is large enough. Here is an example of a disconnected graph with $d=3$: 

\begin{figure}[H]
\scalebox{0.3}{\includegraphics{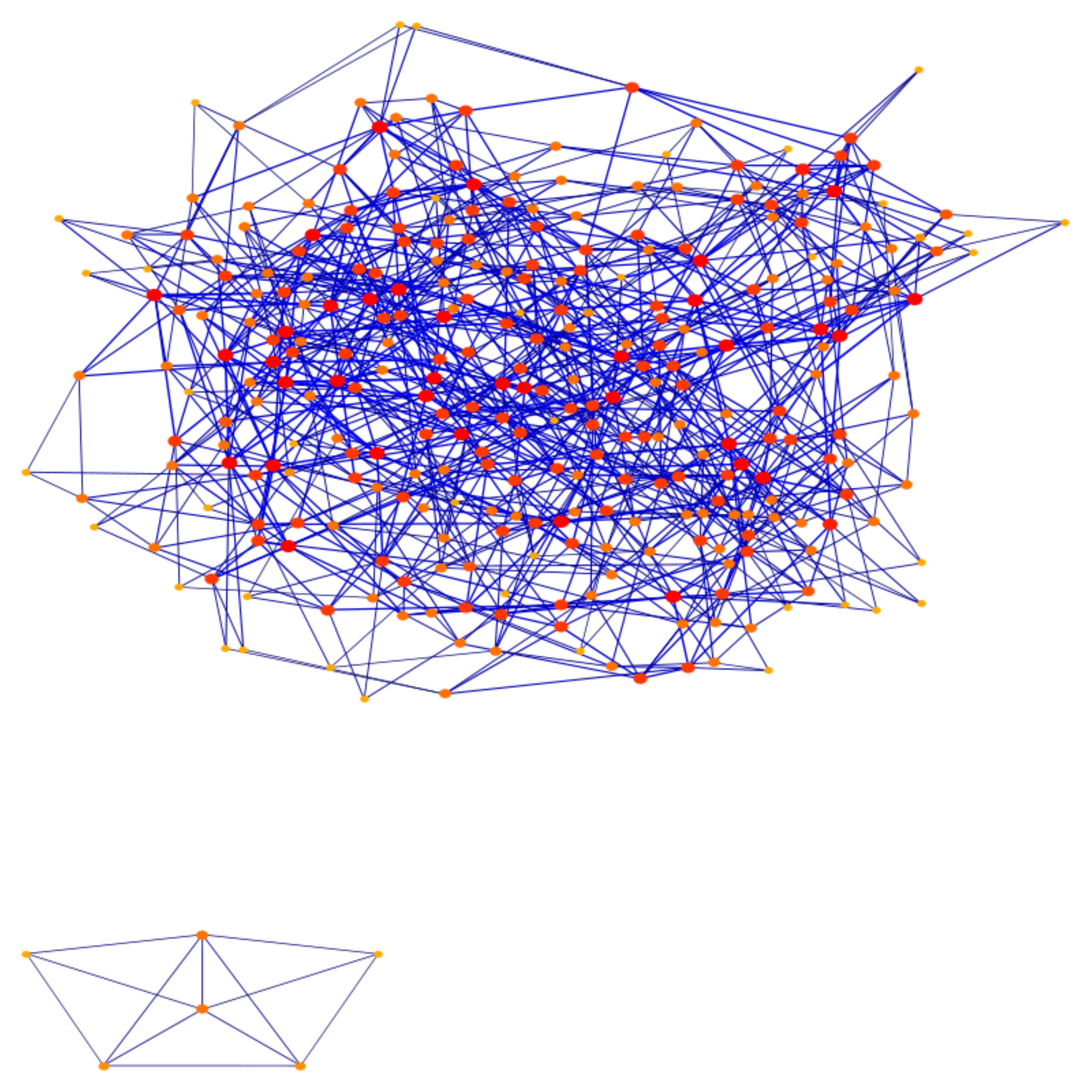}}
\caption{
The graph $(Z_{311},x^2+57,x^2+58,x^2+213)$ is a
quadratic graph with 3 generators over a prime field which is not connected. 
It consists of two different universes. One of them is small: 
it has a subgraph with vertices $\{ 77, 78, 233, 234, 79, 232 \; \}$. 
It is remarkable that this "diamond" is simultaneously invariant under $3$ 
different quadratic maps. We have not found an other one yet. Diamonds seem rare.
} \end{figure}

\question{What is the nature of these exceptions? Are there only finitely many? }

\begin{figure}[H]
\scalebox{0.13}{\includegraphics{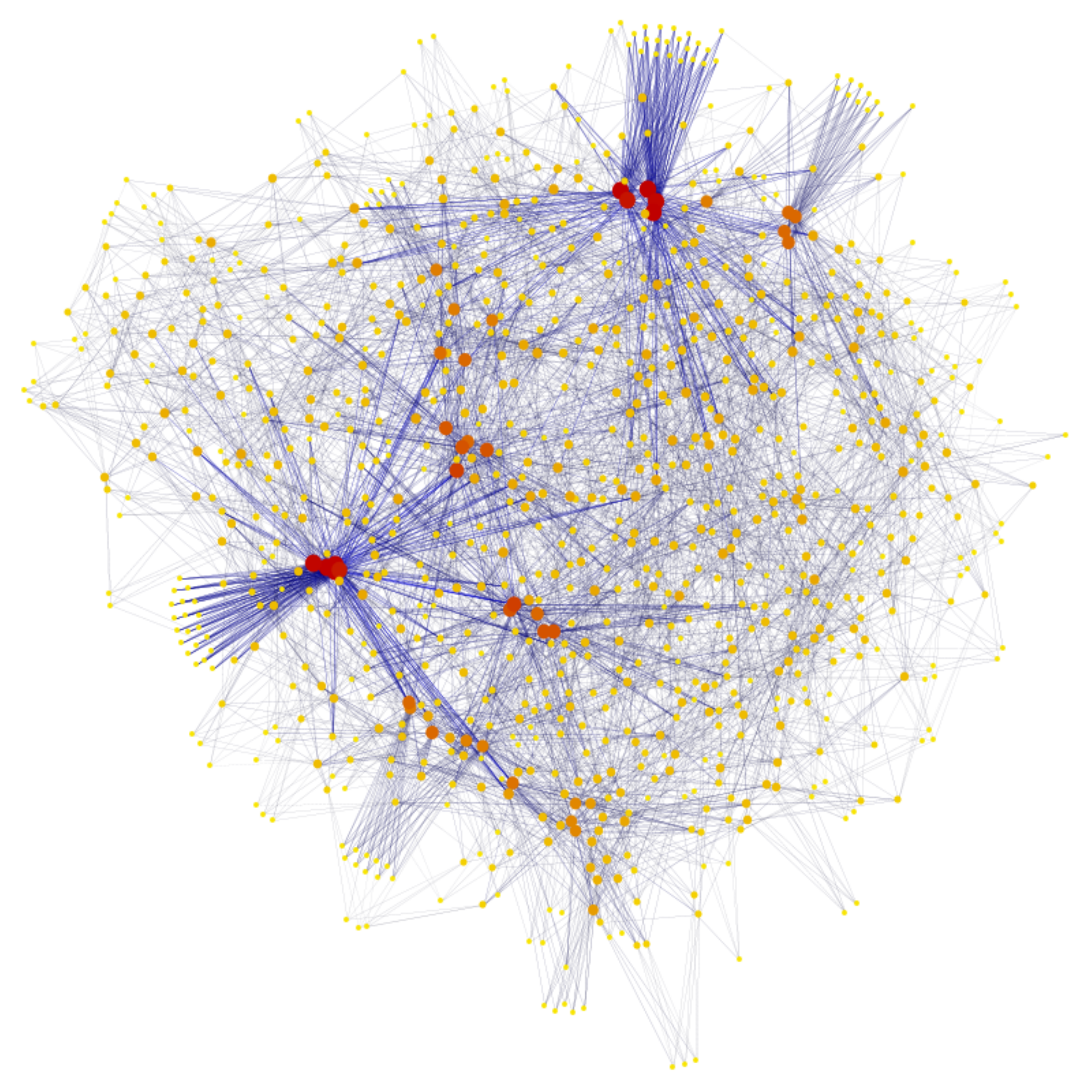}}
\scalebox{0.13}{\includegraphics{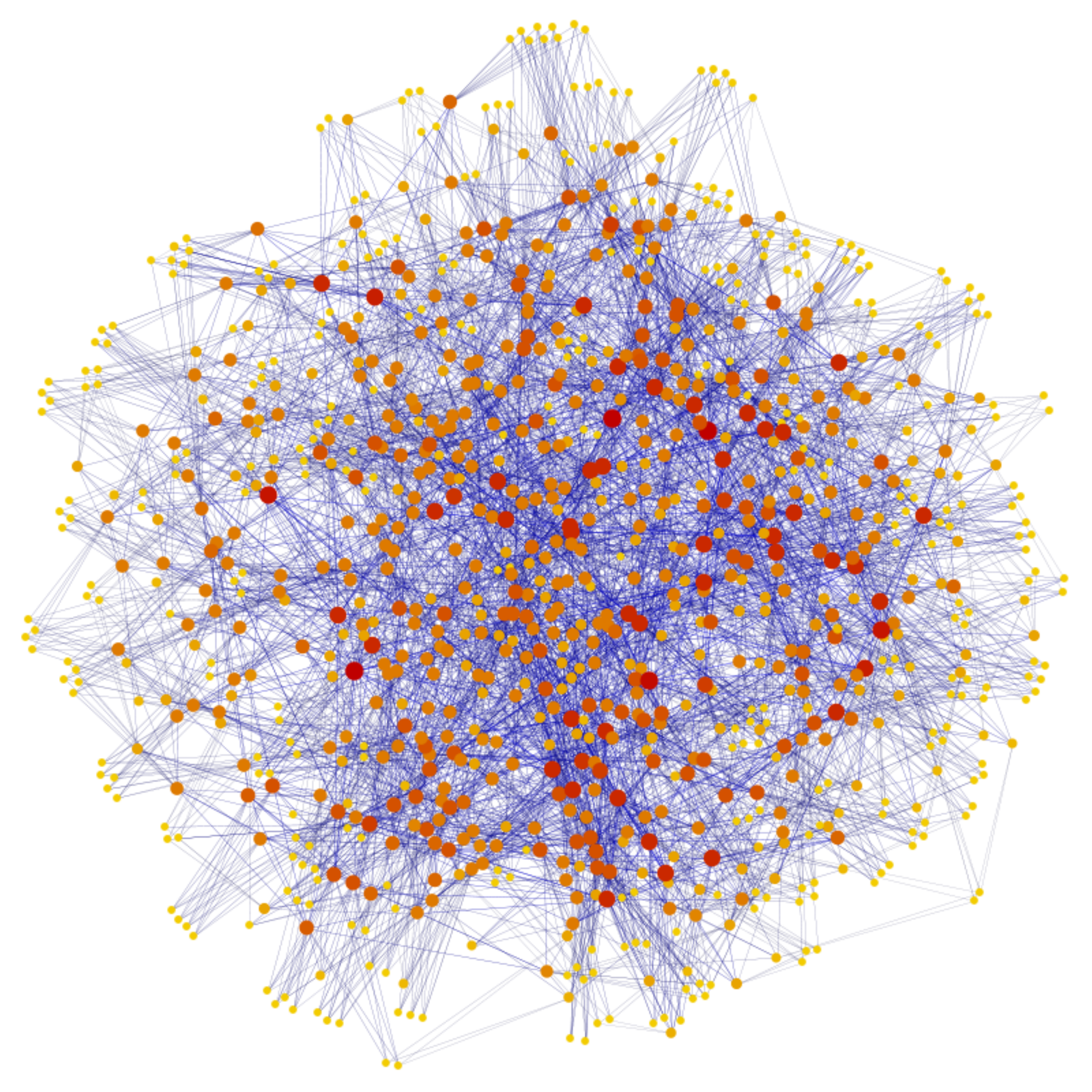}}
\scalebox{0.13}{\includegraphics{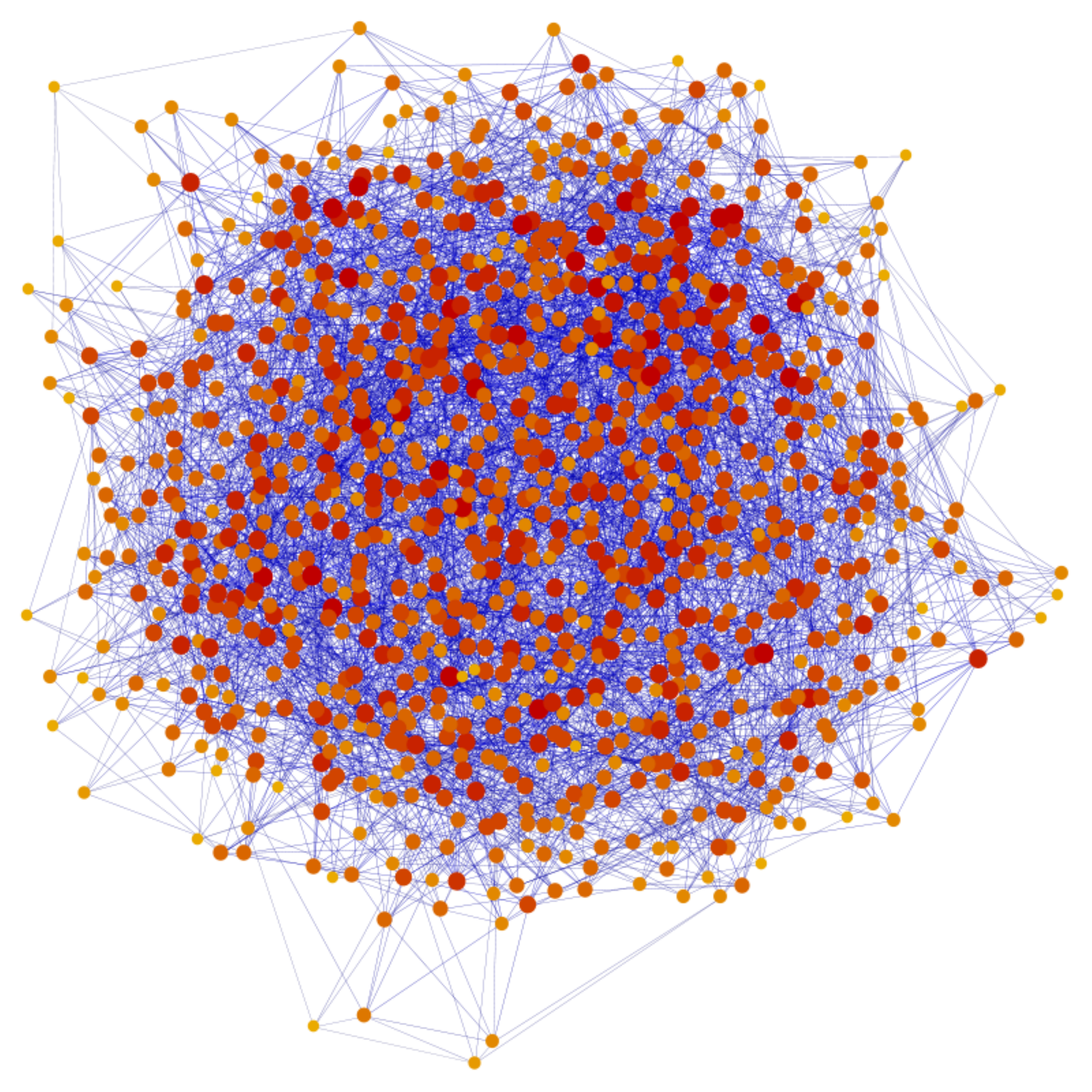}}
\caption{
The orbital network generated by $T_k(x)=x^2+k$ with $k=1,\dots,5$. 
on $Z_{1000}, Z_{1001}$ and $Z_{1009}$. For smooth numbers we see
typical rich club phenomena where a few nodes grab most of the attention
 and are highly connected. In the third case, where $n=1009$ is prime, 
the society is more uniform. 
} \end{figure}

\begin{figure}[H]
\scalebox{0.13}{\includegraphics{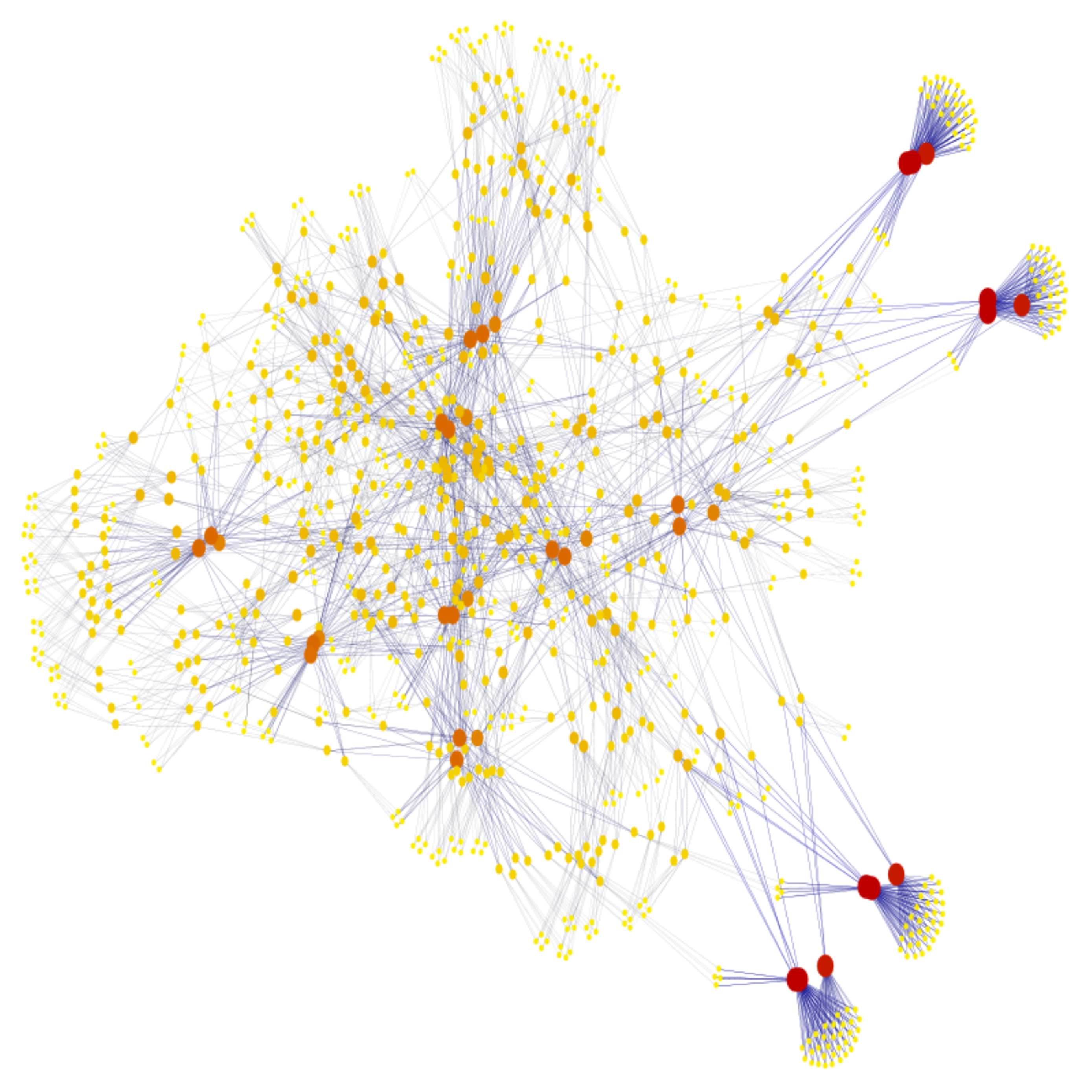}}
\scalebox{0.13}{\includegraphics{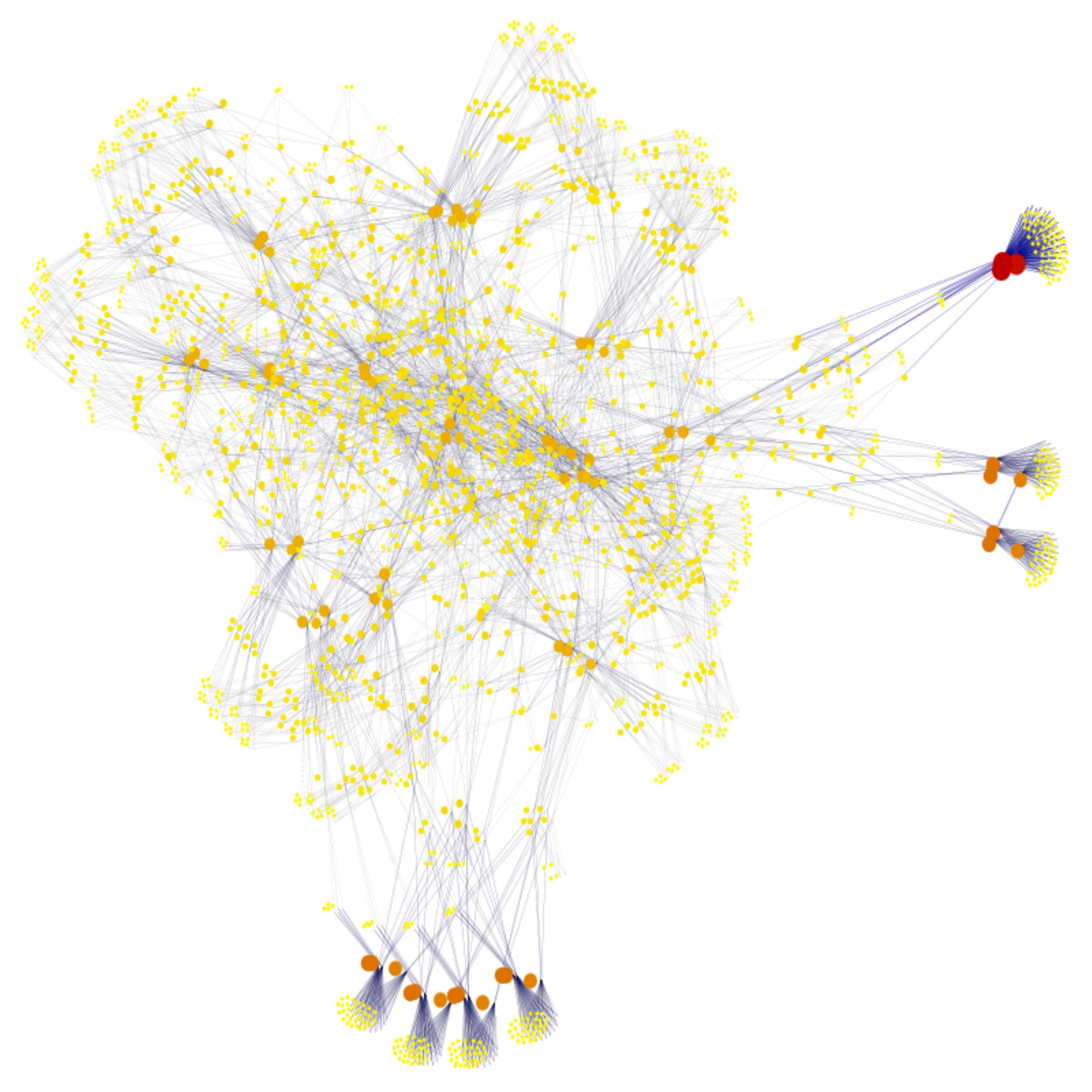}}
\scalebox{0.13}{\includegraphics{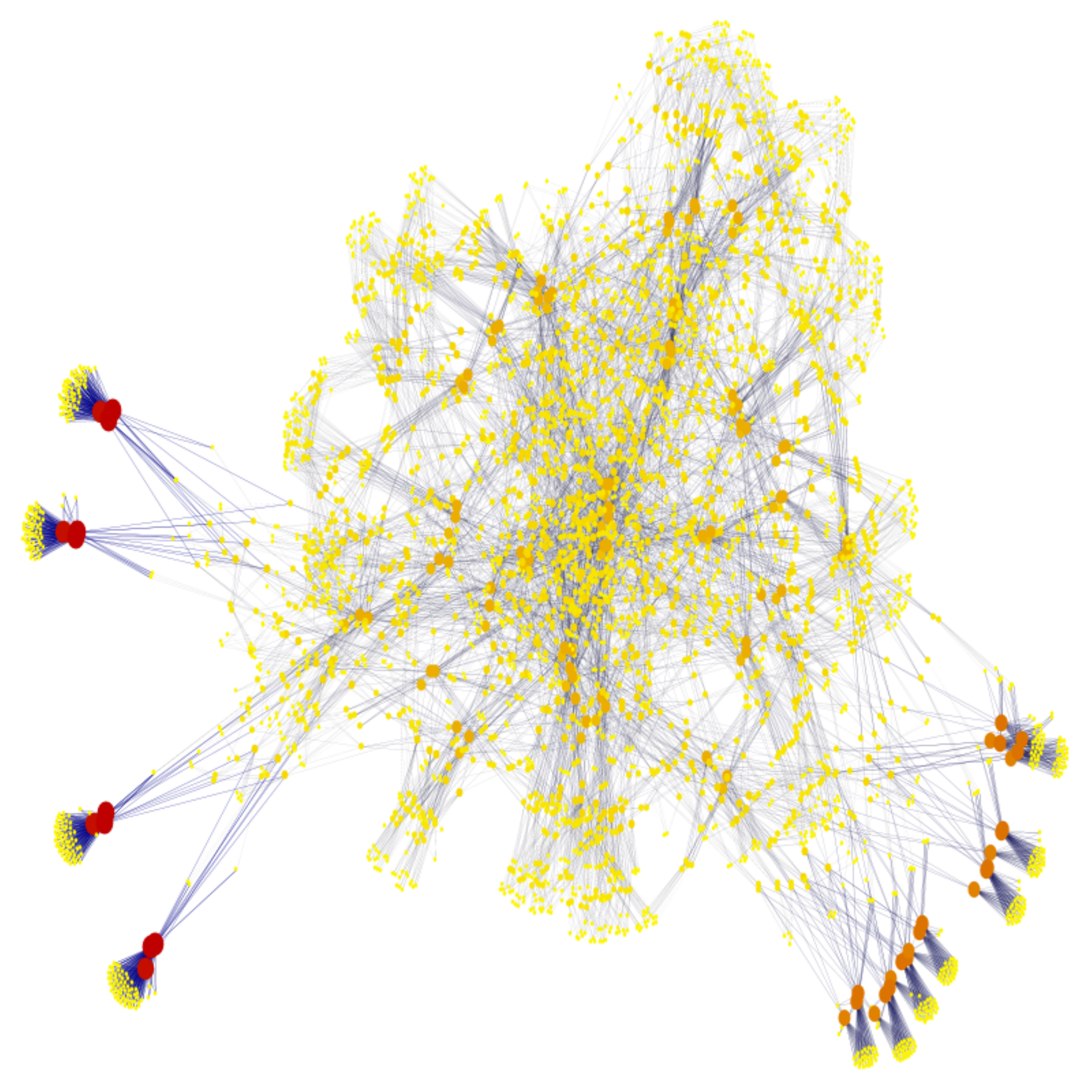}}
\caption{
The orbital network generated by three transformations 
$T_k(x)=x^2+k$ with $k=1,\dots,3$
on $Z_{2^{10}}, Z_{2^{11}}$ and $Z_{2^{12}}$. 
} \end{figure}

\section{The branch graph}

Given a vertex x, call 
$$ B(x) = \{ T^w x \; | \;  w \; {\rm word} \; {\rm in} \; T \; \}  $$
the {\bf branch} generated by $x$. It is the orbit of $x$ under the action of the monoid $T$. 
Speaking in physics terms, it is the {\bf future} of the vertex $x$ because invertibility built into the
monoid $T$ produces an {\bf arrow of time}. \\

Lets call the orbital graph {\bf 1-connected} if there exists 
$x \in R$ such that its branch is $R$. An orbital graph which is not connected needs several
{\bf light sources} to be illuminated completely.
Define a new graph $B(G)$ on R, where two points $x,y$ are connected, if their branches intersect. 
Lets call it {\bf k-connected}, if $k$ is the minimal number of branches reaching $R$. 

Lets call the orbital graph {\bf positively connected} if for every $x,y$, the branches
of $x$ and $y$ intersect.  A positively connected orbital graph is connected. Is the reverse true?

No, there are counter examples as seen in Figure~(\ref{branchcounterexample}). In this case there
are two {\bf universes} or {\bf communities} which are separated in the sense that one can not
get from one to the other by applying generators. While they are connected, each community is 
unreachable in the future from the other community. There is however a third community which 
reaches both. \\

Obviously $G$ is connected if and only if $B(G)$ is connected and by definition, 
$B(G)$ is the complete graph if $G$ is positively connected. 
Here is an orbital graph, where $B(G)$ is not the complete graph:

\begin{figure}[H]
\scalebox{0.2}{\includegraphics{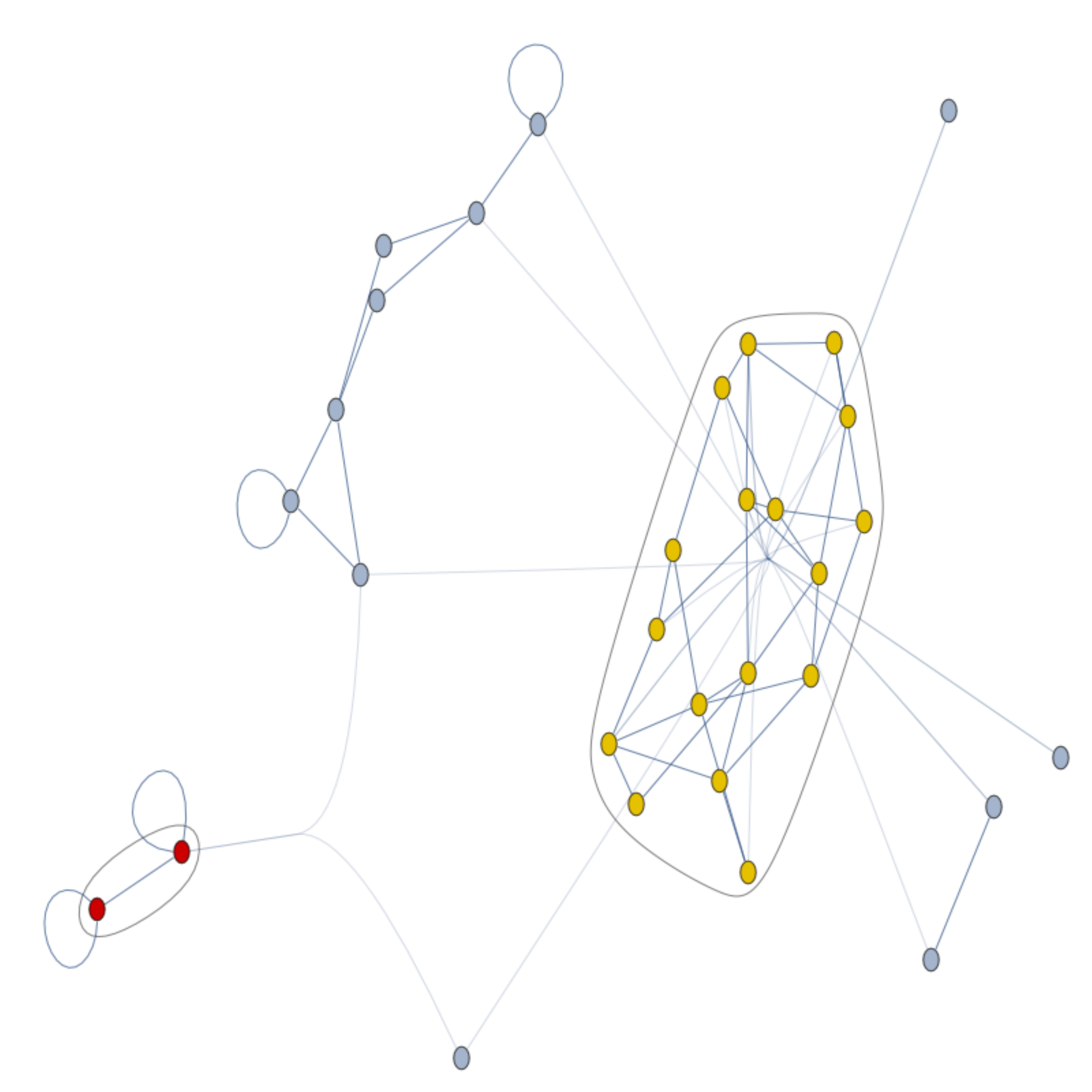}}
\caption{
The graph on $Z_p$ generated by $x^2+4$ and $x^2+7$. 
The branch of the point $x=1$ has length $17$ and is 
disjoint from the branch of the point $14$ which consists
only of two points $15,17$. While the graph is connected,
it is only 3-connected: we can not get out of the two
communities by applying the quadratic maps. This situation
is rather rare. In most cases, two different branches intersect. 
\label{branchcounterexample}
} \end{figure}

Of course, already disconnected graphs produces examples where all branches $B(x)$ are not
the entire graph. 

\section{Euler characteristic}

\begin{figure}[H]
\scalebox{0.2}{\includegraphics{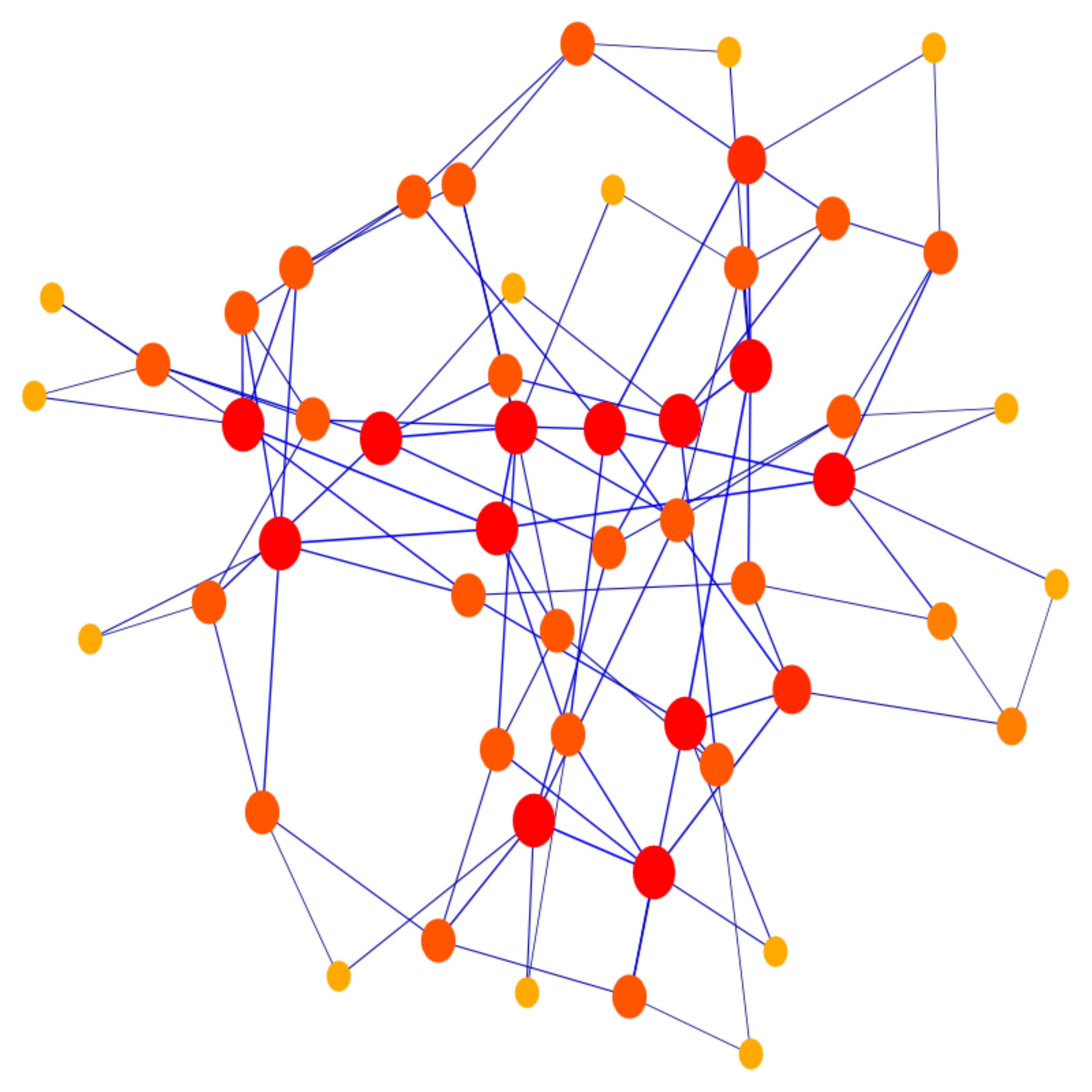}}
\caption{
The graph $(Z_{53},x^2+1,x^2+17)$ has minimal Euler characteristic $-52$ among all quadratic
orbital networks on $Z_{53}$ with two generators.
} \end{figure}

Given a graph $G$, denote by $c_k$ the number of complete subgraphs $K_{k+1}$. 
The number 
$$  \chi(G) = \sum_{k=0}^{\infty} (-1)^k c_k $$ 
is called the {\bf Euler characteristic} of $G$.  
Since in the case $d=1$, we have no tetrahedra, the Euler characteristic is in this case given by 
$$ \chi(G) = v-e+f \; ,  $$
where $v=|V|=c_0$ is the number of vertices, $e=|E|=c_1$ is the number of edges and 
$f=|F|=c_2$ is the number of triangles. We can also write by the Euler-Poincar\'e formula
$$ \chi(G) = b_0-b_1 \; ,  $$
where $b_i$ are the Betti numbers (for cohomology and a proof of Euler-Poincar\'e,
see e.g \cite{knillmckeansinger}).  \\

Here are some observations: 

\begin{lemma}
For any orbital network with $d=1$, the Euler characteristic is nonnegative. 
\end{lemma}
\begin{proof}
The reason is that the graph comes from a directed graph which always has 
maximally one outgoing edge at every vertex. Every forward orbit $T^n(x)$
ends up at a unique attractor. The number of attractors is $b_1$. The number
of components of the graph $b_0$ is clearly larger or equal than $b_1$. 
\end{proof} 

We observe that minimal Euler characteristic is constant $0$ for $p>7$
and networks $T(x)=x^2+a$. 

\begin{lemma}
For $d=2$ and $p$ large enough, the Euler characteristic is always negative. 
\end{lemma}
\begin{proof}
There are only finitely many solutions $C_1$ of Diophantine equations which 
reduce the average degree and there are only finitely many solutions $C_2$ of
Diophantine equations which produce triangles. The Euler characteristic 
is now bounded above by $p-2p+C_1 + C_2$ which is negative for large enough $p$. 
\end{proof} 

Remark. We see that for $p>13$ the Euler characteristic is always negative. \\

For $d=2$, the minimum is always very close to $-p$. For $p>23$ we see already 
that the minimal Euler characteristic is always either $1-p$ or $2-p$.
The maximum is $0$ for $p=17$ and becomes negative afterwards. For $p=29$ for
example, it is $-11$. We see also the difference between the maximum and minimum 
seem to settle pretty much.  The minimum is $-p$ or $-p+1$ and the maximum 
between $23-p$ or $25-p$. \\

\begin{figure}[H]
\scalebox{0.2}{\includegraphics{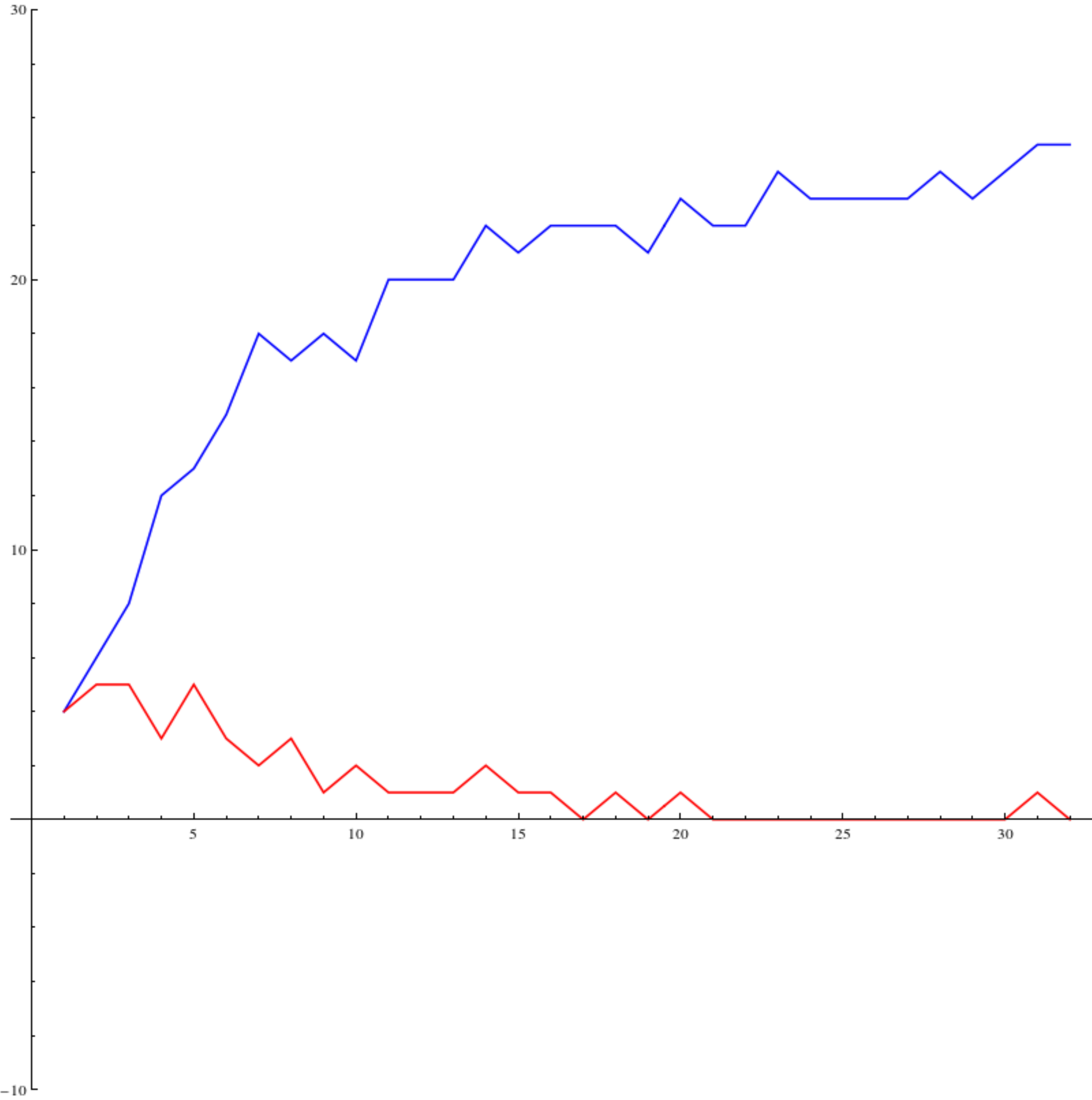}}
\caption{
We see the minimal and maximal Euler characteristic for all quadratic orbital graphs
with two generators on $Z_p$ as a function of primes $p$ up to $p=137$. The figure shows
$\chi_{\rm min}(G(Z_{p}))+p$ and $\chi_{\rm max}(G(Z_{p}))+p$ which prompts the question
whether these numbers will settle eventually. 
} \end{figure}

\begin{lemma}
If a graph has no triangles and uniform degree $4$ then the Euler 
characteristic is $-p$. 
\end{lemma}
\begin{proof}
The handshaking lemma telling that the average degree is $2 |E|/|V|$
which means that $|E|=2|V|$. Because there are no triangles, we know
that $\chi(G) =|V|-|E| = -|V|=-p$. 
\end{proof} 

The uniform degree and no triangle case is "generic" in the sense that
it happens if there are no solutions to Diophantine equations like 
$T^2=T, T^3=Id, T^2=Id$. The question is whether we always can find $a$
which avoids these cases. This leads to the question: 

\question{Is the minimal Euler characteristic equal to  $-p$ for large $p$?}

\question{Is the difference between minimal and maximal Euler characteristic constant for large $p$? }

The minimal Euler characteristic has led for almost all larger $p$ cases to regular graphs. 
They are achieved if degree reducing Diophantine equations like 
$T(S(x))=x$,$S(T(x))=x$,$T^2(x)=x$,$T(x)=S(x)$,$S^2(x)=x$ have no solutions
and triangle producing Diophantine equations like $TST(x)=x$,$STS(x)=x$
$S^2T(x)=x$,$T^2S(x)=x$,$T S^2(x)=x$,$ST^2(x)=x$,$S(x)=T^2(x)$,$T(x)=S^2(x)$ 
have no solutions. These are all higher degree congruences. Heuristically, if 
each of these $K$ equations have no solution with probability $1/2$, then we
expect for $p^2>2^(-K)$ to have cases where none has a solution. In any case, we
expect for large enough $p$ to have cases of regular graphs 
with Euler characteristic $-p$. As always with Diophantine equations, this could
be difficult to settle. 

\section{The number of cliques}

For $d=1$ we see no tetrahedra but in general 2 triangles. This can be proven: 

\begin{lemma}
For $d=1$, there can not be more than 2 triangles.
\end{lemma}
\begin{proof}
Triangles are solutions to the Diophantine equation $T^3(x)=x$ without satisfying 
the Diophantine equation $T^2(x)=T(x))$ or $T^2(x)=x$.  For $T(x)=x^2+a$ we have 
maximally $8$ solutions to 
$$  T^3(x) - x = a^4+4 a^3 x^2+2 a^3+6 a^2 x^4+4 a^2 x^2+a^2+4 a x^6+2 a x^4+a+x^8 - x  \;  $$
and
$$  T^2(x)-T(x) = x^4+5x^2+9  \; . $$
There are maximally $8$ solutions to the first equation. This means that we have either $1$
or $2$ triangles. Because triangles form cycles of $T$ they can not be adjacent but have to 
consist of disjoint vertices. 
\end{proof} 

{\bf Remarks}. \\
{\bf 1)} With primes $p$ larger than $19$ we so far always have found
a map $T(x)=x^2+a$ on $Z_p$ for which we have $2$ triangles. \\
{\bf 2)} For cubic $x^3+a$ have less or equal than  $3^3/3=9$ triangles.
For $p=53$ and $a=0$ we have $8$ triangles. All other cases with $T(x)=x^3+a$
on $Z_p$ we have seen has less or equal than $2$ triangles. The case $p=53$
is special because $x^{27}-x$ has $27$ solutions modulo $53$ 
by {\bf Fermat's little theorem}. 

\begin{lemma}
For $d=1$ and prime $p$, there are no $K_4$ graphs. 
\end{lemma}
\begin{proof}
A tetrahedral subgraph would have $4$ triangles.
\end{proof}

{\bf Remark.}
Here is a second proof which works for higher order polynomials too: 
look at the directed graph with edges $(x,y)$ if $T(x)=y$. 
The degree of each edge can not be larger than $3$ and we 
have maximally one outgoing edge at each point.
Lets look at a tetrahedron and ignore connections to it.
Let $V_{\pm}(x)$ the number of out and incoming directions at 
a node $x$. We must have $4=\sum_x V_+(x) = \sum_x V_-(x)=8$,
a contradiction showing that the tetrahedron is not possible.  \\

For two generators $d=2$, the questions become harder. The maximal number of triangles
we have seen is $18$, which happens for $p=67$. \\

\begin{lemma}
There are no $K_6$ subgraphs for $d=3$. 
\end{lemma}
\begin{proof} 
Look at such a subgraph and the directed graph. This 
is not possible because we have only 2 outgoing edges
at each vertex. 
\end{proof}

\question{Can there be $K_5$ subgraphs in any of the spaces $X_p^d$? }

We have checked for all $p \leq 47$. If the answer is $no$, and
for $d=2$ there are no graphs $K_5$ embedded, then the Euler characteristic
is $v-e+f-t$ where $v$ is the number of vertices, $e$ the number of edges and $f$ the number of 
triangles and $t$ the number of tetrahedra.  \\

\question{What is the maximal number of $K_4$ graphs for $d=2$?}

It is very small in general. We have never seen more than $2$ 
(which was the case for $p=19)$.

\section{Diameter}

The minimal diameter among all graphs on $X_p^2$ seems to be monotone in $p$. 
$3$ is the largest with diameter $1$, $7$ the largest with minimal diameter $2$
$13$ the largest with diameter $3$, $31$ the largest minimal diameter $4$. 
$61$ the largest with diameter $5$, $127$ the largest with minimal diameter $6$.
For $241$ the largest diameter is $7$. For $p=251$ it is already $8$ even so $251<2^8$. \\

While we see in experiments that diameter of a network in $X_p^2$ can not be smaller than $\log_2(p)$
we can only show a weaker result: 

\begin{lemma}
The diameter in $X_p^2$ can not be smaller than $\log_6(p)$. 
\end{lemma}
\begin{proof}
The maximal degree is $6$. A spanning tree has diameter less than $\log_6(p)$. 
\end{proof}

\question{Can the diameter in $X_p^2$ become smaller than $\log_2(p)$?}

\section{Planar graphs}

We measure that for $p>23$, quadratic graphs with two generators are not planar. If
there are no graphs $K_5$ embedded, we would by Kuratowski's theorem know that they 
must contain homeomorphic stretched utility subgraphs $K_{3,3}$. 

\begin{figure}[H]
\scalebox{0.2}{\includegraphics{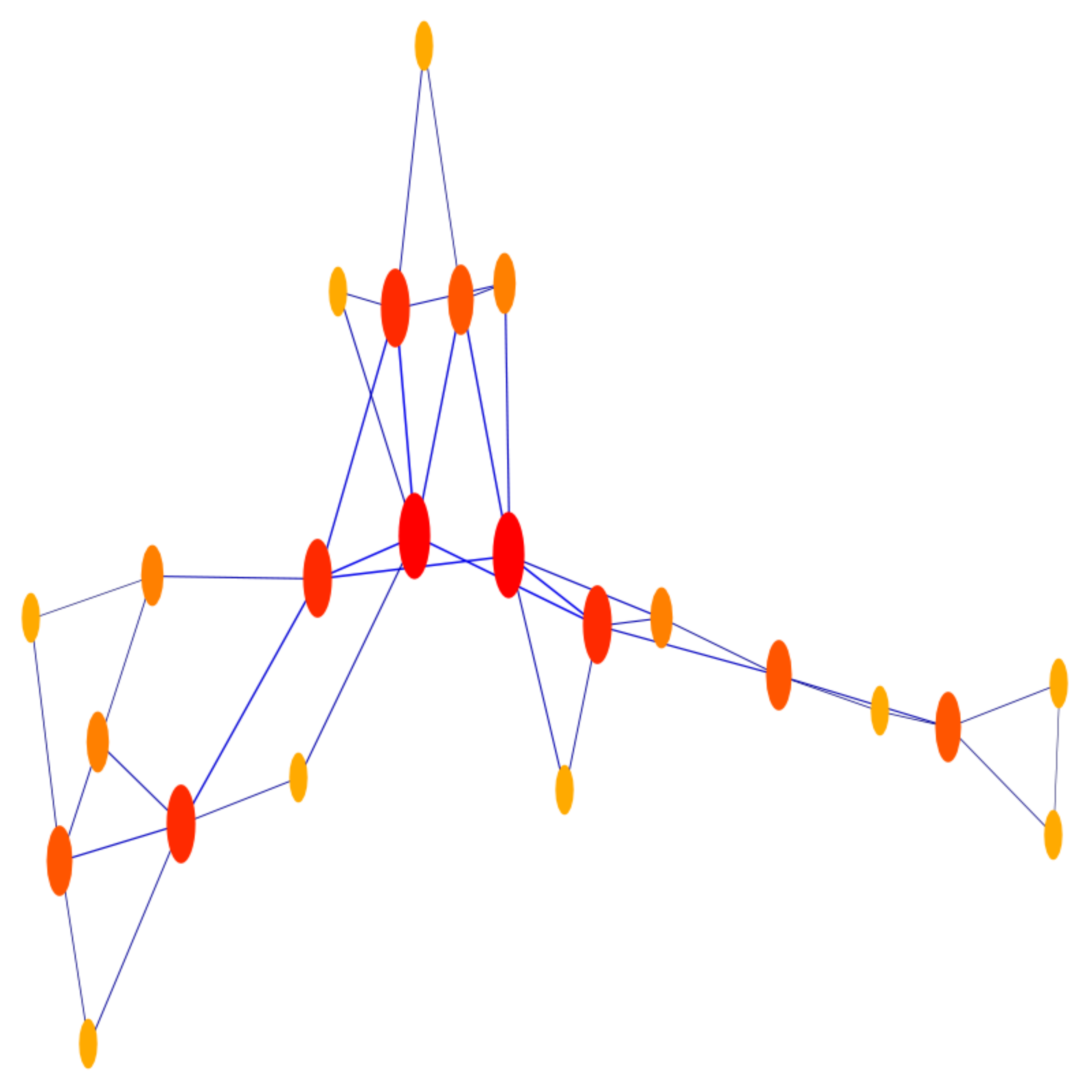}}
\scalebox{0.2}{\includegraphics{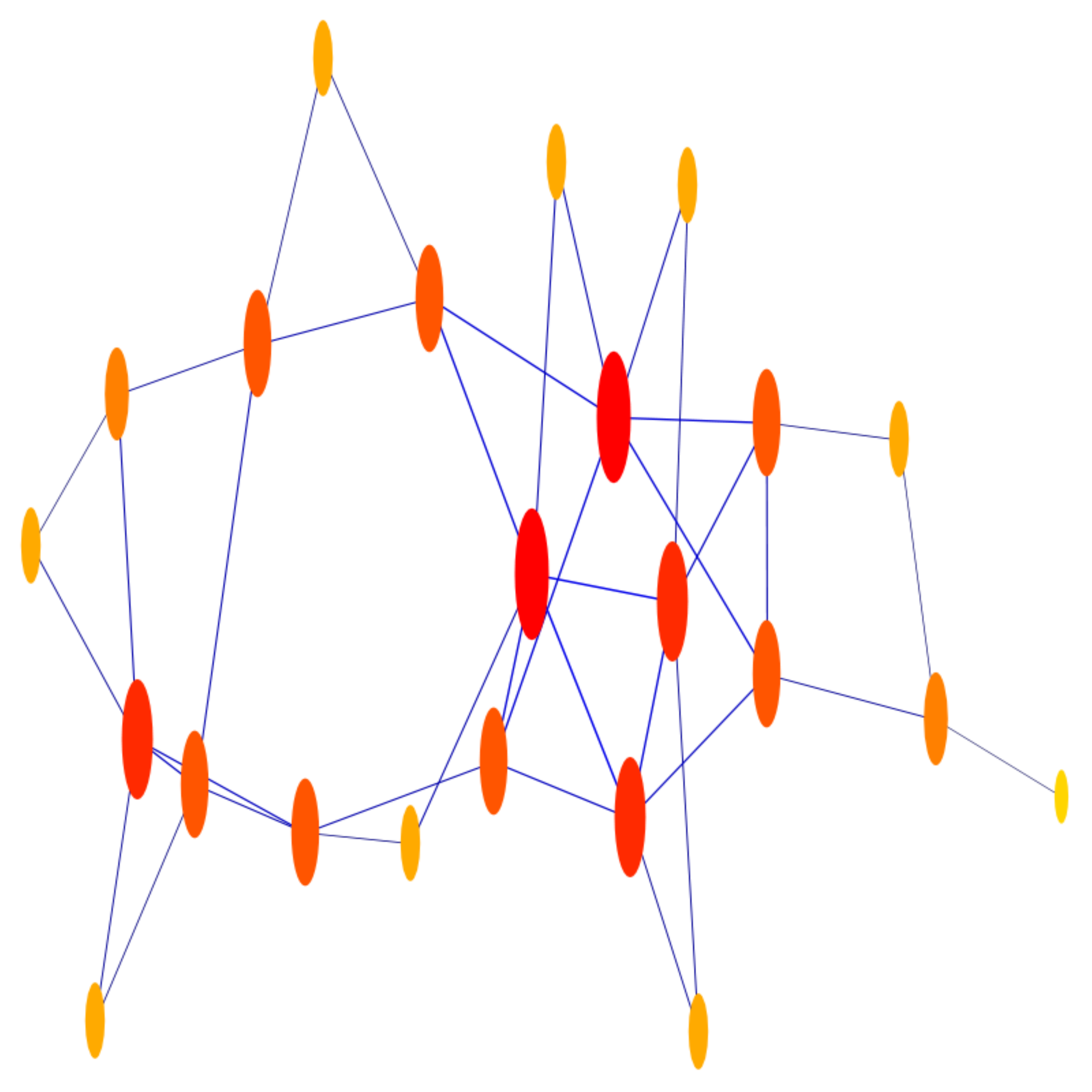}}
\caption{
The graph $(Z_{23},x^2+4,x^2+20)$ and
$(Z_{23},x^2+11,x^2+17)$ quadratic networks which are planar. 
They might be the largest quadratic planar graphs with two 
or more different generators. 
} \end{figure}

The largest to us known planar quadratic graph with two generators is $p=23$.

\question{Is there a planar quadratic orbital network in $X_p^2$  for $p>23$? }

\section{Dimension}

The inductive dimension for graphs \cite{elemente11} is formally related to the inductive 
Brouwer-Menger-Urysohn dimension for topological spaces. Let $S(x)$ denote the unit sphere of a vertex
$x$. Define ${\rm dim}(\emptyset)= -1$ and inductively: 
$$ {\rm dim}(G) = 1+\frac{1}{|V|} \sum_{v \in V} {\rm dim}(S(v))  \; ,  $$
where $S(x)$ is the graph generated by vertices connected to $x$. Isolated points have dimension $0$,
trees with at least one edge or cycle graphs with at least $4$ vertices 
have dimension $1$, a graph $K_{n+1}$ has dimension $n$. \\

Since there are no isolated points, the dimension of quadratic orbital graphs is always $\geq 1$ with 
the exception of $T(x)=x^2$ on $Z_2$. We see that the minimum $1$ 
is attained for all $p>11$ and that the maximum can become larger than $2$ for $p=7$, where the
maximum is $449/210$ for $a=2,b=3$. This seems to be the only case.  \\

\begin{lemma}
For every $p,k,d$ there exists a constant $C=C_m(d) \leq d^{m+1} 2^{m+1}$ 
such that the number of $K_{m+1}$ subgraphs of $G \in X_p^d$ is smaller than $C$. 
\end{lemma}
\begin{proof}
The existence of a subgraph is only possible if one of finitely many
systems of finitely many Diophantine equations consisting of polynomials 
has a solution. Given $Z_n$ and the degree of the generating polynomials,
the constant $d$ gives an upper bound for the degree of polynomials involved.
Since a polynomial $f$ in $Z_n$ has less or equal solutions than the degree
and the degree of the polynomials involved is less or equal to $2^{m+1}$
and there are less than $d^{m+1}$ polynomials which can occur.
\end{proof}

Of course, this applies for any polynomial map. For non-prime $p$ we can 
have more solutions. But since 
the number $A(n)$ of solutions to a polynomial equation $f(x)=0 {\rm mod} \;  n$ 
is multiplicative (Theorem 8.1 in \cite{Hua}), the following result should be
true for all $n$. We prove it only for primes: 

\begin{lemma}
For all quadratic networks the dimension goes to $1$ for $n \to \infty$ along primes. 
\end{lemma}
\begin{proof}
We have seen that for fixed $n$ and $d$ quadratic polynomial maps
there is a bound $C_m(n,d)$ such 
that the quadratic orbital graph on $Z_n$ has less or equal $C_m$ 
sub graphs $K_{m+1}$. This means that the dimension of a vertex can be
$m$ only for a bounded number of vertices. Since for $n \to \infty$ the
number of vertices goes to infinity, the inductive dimension goes to zero. 
\end{proof}

A modeling question is: 

\question{How large does $d$ and $p$ have to 
be chosen in order to get an orbital network of given dimension?}

\begin{figure}[H]
\scalebox{0.2}{\includegraphics{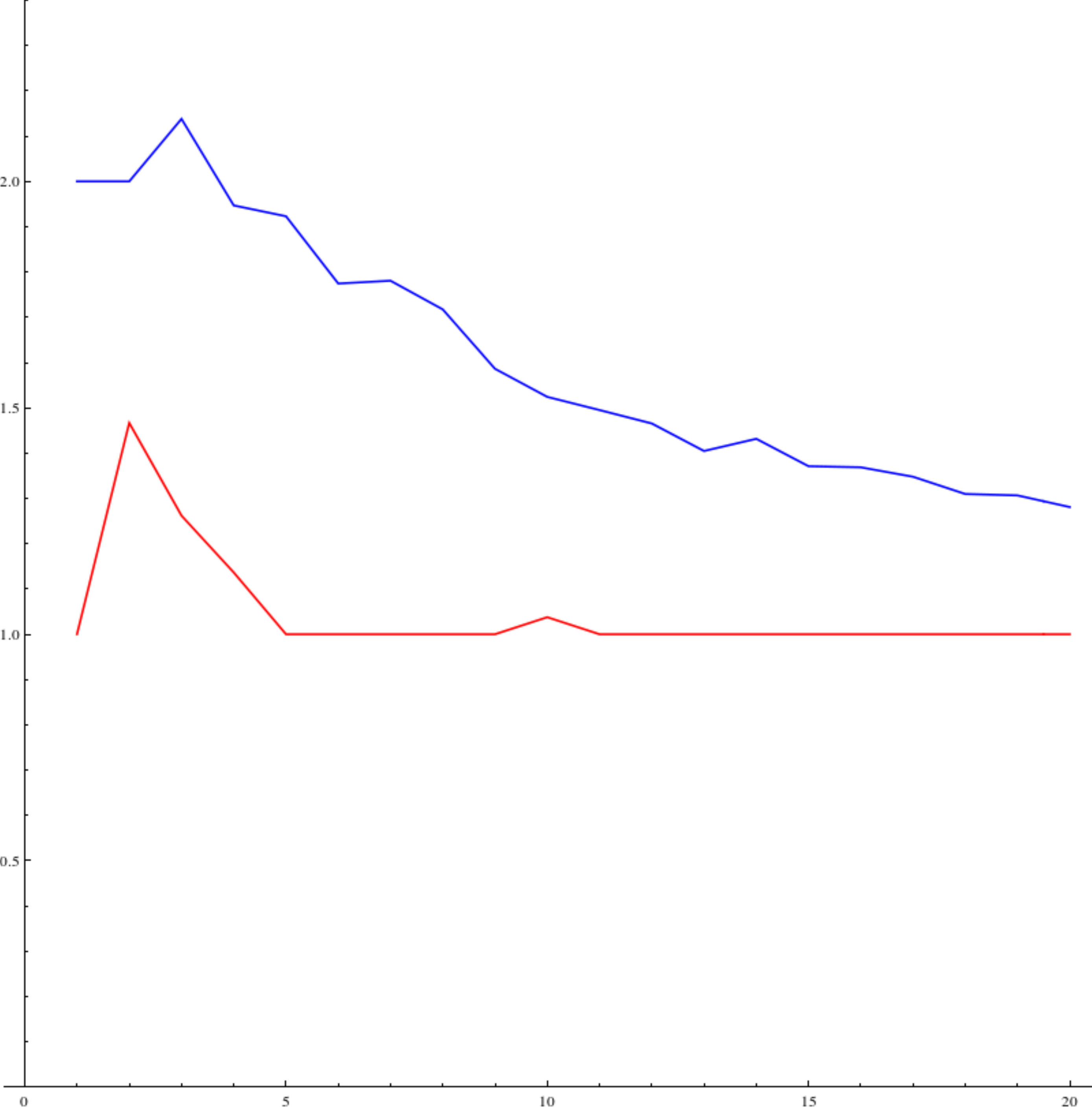}}
\caption{
The minimal and maximal dimension of quadratic orbital graphs with two generators 
on $Z_p$ for $p \leq 73$. 
} \end{figure}

\vspace{12pt}
\bibliographystyle{plain}

\end{document}